\documentclass[letterpaper,11pt]{amsart}

\usepackage{amsfonts,amsmath,amssymb,amsthm}

\newtheorem{thm}{Theorem}[section]
\newtheorem{prop}[thm]{Proposition}
\newtheorem{lem}[thm]{Lemma}
\newtheorem{cor}[thm]{Corollary}

\newtheorem{asm}{Assumption}

\theoremstyle{remark}
\newtheorem{rem}[thm]{Remark}

\theoremstyle{definition}

\newcommand{\ra}{\rightarrow}

\newcommand{\Lra}{\Longrightarrow}

\newcommand{\R}{\mathbb R}     
\newcommand{\Z}{\mathbb Z}     

\renewcommand{\a}{\alpha}
\renewcommand{\b}{\beta}

\renewcommand{\d}{\delta}

\newcommand{\e}{\varepsilon}

\renewcommand{\l}{\lambda}

\newcommand{\s}{\sigma}

\renewcommand{\k}{\kappa}

\newcommand{\limn}{\underset{n\rightarrow\infty}{\longrightarrow}}

\newcommand{\fl}[1]{\lfloor #1 \rfloor}  
\newcommand{\ind}[1]{ \mathbf{1}_{ \{ #1 \} } } 

\newcommand{\be}{\begin{equation}}
\newcommand{\ee}{\end{equation}}
\DeclareMathOperator{\Var}{Var}   \DeclareMathOperator{\Cov}{Cov}  
   
\def\Pv{\mathbf{P}}  \def\Ev{\mathbf{E}}  \def\Varv{\mathbf{Var}} 

\def \eil{\stackrel{{\rm law}}{=}} 
\def\ProbSpace{\bigl( \Omega, {\mathcal F}, \Pv\bigr)} 
\def\one{{\bf 1}}

\newcommand{\w}{\omega}              
\renewcommand{\P}{\mathbb{P}}        
\newcommand{\E}{\mathbb{E}}          
\newcommand{\vp}{\mathrm{v}_P}       

\newcommand{\tc}{C_0}		
\newcommand{\nc}{\bar{\a}}		

\begin{document}

\title[Weak quenched limits]{Weak quenched limiting distributions for transient one-dimensional random walk in a random environment}
\author{Jonathon Peterson}
\address{Jonathon Peterson \\  Cornell University \\ Department of Mathematics \\ Malott Hall \\ Ithaca, NY 14853 \\ USA}
\curraddr{Purdue University \\ Department of Mathematics \\ 150 N. University Street \\ West Lafayette, IN 47907 \\ USA}
\email{peterson@math.purdue.edu}
\urladdr{http://www.math.purdue.edu/~peterson}
\thanks{J. Peterson was partially supported by National Science Foundation grant DMS-0802942.}

\author{Gennady Samorodnitsky}
\address{Gennady Samorodnitsky \\ Cornell University \\ School of Operations Research and Information Engineering \\ Ithaca, NY 14853 \\ USA}
\email{gennady@orie.cornell.edu}
\urladdr{http://legacy.orie.cornell.edu/~gennady/}
\thanks{G. Samorodnitsky was partially supported by ARO grant
  W911NF-10-1-0289  and NSF grant DMS-1005903 at Cornell University}


\date{\today}

\begin{abstract}
We consider a one-dimensional, transient random walk in a random
i.i.d. environment. The asymptotic behaviour of such
random walk depends to a large extent on a crucial parameter
$\kappa>0$ that determines the fluctuations of the process. 
When $0<\k<2$, the averaged distributions of the hitting times of the random walk converge to a $\kappa$-stable distribution. However, it was shown recently that in this case there does not exist a quenched limiting distribution of the hitting times.
That is, it is not true that
for almost every fixed environment, the 
distributions of the hitting times (centered and scaled in any manner)
converge to a non-degenerate distribution. We
show, however, that the quenched distributions do have a limit in the
weak sense. That is, the quenched distributions of the hitting times 
-- viewed as a random probability measure on $\R$ -- converge in
distribution to a random probability measure, which has interesting
stability properties.  Our results generalize both the averaged limiting distribution and the non-existence of quenched limiting distributions. \\

 Nous consid\'erons une marche al\'eatoire unidimensionnelle dans un environnement i.i.d. Le comportement asymptotique d'une telle marche al\'eatoire d\'epend largement d'un param\`etre crucial $\kappa$ qui d\'etermine les fluctuations du processus. Si $0< \kappa <2$, alors les distributions moyennis\'ees des temps d'atteinte de la marche al\'eatoire convergent vers une loi $\kappa$-stable. Cependant, il a \'et\'e r\'ecemment prouv\'e que dans ce cas l\`a, il n'existe pas de distribution limite des temps d'atteinte \`a environnement fix\'e. C'est-\`a-dire, il n'est pas vrai que presque tout environnement fix\'e , les distributions des temps d'atteinte (centr\'es et normalis\'es de quelque mani\`ere que ce soit) convergent vers une distribution non d\'eg\'en\'er\'ee. Nous montrons n\'eanmoins que les distributions \`a environnement fix\'e ont une limite au sens faible.  Plus pr\'ecis\'ement,  les distributions \`a environnement fix\'e des temps d'atteinte -- vues comme des mesures de probabilit\'e al\'eatoires sur $\mathbb{R}$ -- convergent en distribution vers une mesure de probabilit\'e al\'eatoire qui a d\'int\'eressantes propri\'et\'es de stabilit\'e. Nos r\'esultats g\'en\'eralisent \`a la fois la limite des distributions moyennis\'ees et la non existence de distributions limites \`a environnement fix\'e.
\end{abstract}

\maketitle

\section{Introduction} \label{sec:intro}

A random walk in a random environment (RWRE) is a Markov chain with
transition probabilities that are chosen randomly ahead of time. The
collection of transition probabilities are referred to as the
\emph{environment} for the random walk. We will be concerned with
nearest-neighbor RWRE on $\Z$, in which case the space of environments
may be identified with $\Omega = [0,1]^\Z$, endowed with the
cylindrical $\sigma$-field. 
Environments $\w=\{\w_x\}_{x\in\Z} \in \Omega$ are chosen according to a probability measure $P$ on $\Omega$. 

Given an environment $\w = \{ \w_x \}_{x\in\Z} \in \Omega$ and an initial location $x\in\Z$, we let $\{X_n\}_{n\geq0}$ be the Markov chain with law $P_\w^x$ defined by $P_\w^x(X_0 = x) = 1$, and
\[
 P_\w^x\left( X_{n+1} = z \, | \, X_n = y \right) = 
\begin{cases}
 \w_y & z= y+1 \\
 1-\w_y & z=y-1 \\
 0 & \text{otherwise}. 
\end{cases}
\]
Since the environment $\w$ is random, $P_\w^x(\cdot)$ is a random probability measure and is called the \emph{quenched} law. By averaging over all environments we obtain the \emph{averaged} law 
\[
 \P^x(\cdot) = \int_\Omega P_\w^x(\cdot) \, P(d\w). 
\]
Since we will usually be concerned with RWRE starting at $x=0$, we will denote $P_\w^0$ and $\P^0$ by $P_\w$ and $\P$, respectively. 
Expectations with respect to $P$, $P_\w$, and $\P$ will be denoted by $E_P$, $E_\w$ and $\E$, respectively. Throughout the paper we will use $\Pv$ to denote a generic probability law, separate from the RWRE, with corresponding expectations $\Ev$. 

We will make the following assumptions on the distribution $P$ on environments
\begin{asm}\label{iidasm}
 The environments are i.i.d. That is, $\{\w_x\}_{x\in\Z}$ is an
 i.i.d.\ sequence of random variables under the measure $P$. 
\end{asm}
\begin{asm}\label{tasm}
 The expectation $E_P[ \log \rho_0 ]$ is well defined and 
$E_P[ \log \rho_0 ]< 0$. Here $\rho_i = \rho_i(\w) =
 \frac{1-\w_i}{\w_i}$, for all $i\in\Z$.  
\end{asm}
In Solomon's seminal paper on RWRE \cite{sRWRE}, he showed that Assumptions \ref{iidasm} and \ref{tasm} imply that the RWRE is transient to $+\infty$. That is, $\P(\lim_{n\ra\infty} X_n = +\infty ) = 1$. Moreover, Solomon also proved a law of large numbers with an explicit formula for the limiting velocity $\vp = \lim_{n\ra\infty} X_n/n$. Interestingly, $\vp>0$ if and only if $E_P[\rho_0] < 1$, and thus one can easily construct examples of RWRE that are transient with ``zero speed.'' 

Soon after Solomon's original paper, Kesten, Kozlov, and Spitzer
\cite{kksStable} analyzed the limiting distributions of transient RWRE
under the following additional assumption.  
\begin{asm}\label{kasm}
 The distribution of $\log \rho_0$ is non-lattice under $P$, and there exists a $\kappa > 0$ such that $E_P[ \rho_0^\kappa ]= 1$ and $E_P[ \rho_0^\k \log \rho_0 ] < \infty$. 
\end{asm}
Kesten, Kozlov, and Spitzer obtained limiting distributions for the random walk $X_n$ by first analyzing the limiting distributions of the hitting times 
\[
 T_x := \inf \{ n\geq 0: X_n = x \}. 
\]
Let $\Phi(x)$ be the distribution function of the standard normal
distribution, and let $L_{\kappa,b}(x)$ be the distribution function
of a totally skewed to the right stable istribution of index $\kappa
\in(0,2)$ with scaling parameter $b>0$ and zero shift; see
\cite{samorodnitsky:taqqu:1994}. 
\begin{thm}[Kesten, Kozlov, and Spitzer \cite{kksStable}]\label{kksthm}
Suppose that Assumptions \ref{iidasm} - \ref{kasm} hold, and let
$x\in\R$. 
\begin{enumerate}
 \item If $\kappa \in(0,1)$, then there exists a constant $b>0$ such that
\[
 \lim_{n\ra\infty} \P\left( \frac{T_n}{n^{1/\kappa}} \leq x \right) = L_{\kappa,b}(x).
\]
 \item If $\kappa = 1$, then there exist constants $A,b>0$ and a
   sequence $D(n) \sim A\log n$ so that
\[
 \lim_{n\ra\infty} \P\left( \frac{T_n-n D(n)}{n} \leq x \right) = L_{1,b}(x). 
\]
 \item If $\kappa \in (1,2)$, then there exists a constant $b>0$ such that 
\[
 \lim_{n\ra\infty} \P\left( \frac{T_n - n/\vp}{n^{1/\kappa}} \leq x \right) = L_{\kappa,b}(x).
\]
 \item If $\kappa = 2$, then there exists a constant $\s>0$ such that 
\[
\lim_{n\ra\infty} \P\left( \frac{T_n - n/\vp}{\s \sqrt{n\log n}} \leq x \right) = \Phi(x).
\]
 \item If $\kappa > 2$, then there exists a constant $\s>0$ such that 
\[
\lim_{n\ra\infty} \P\left( \frac{T_n - n/\vp}{\s \sqrt{n}} \leq x \right) = \Phi(x).
\]
\end{enumerate}
\end{thm}
Theorem \ref{kksthm} is then used in \cite{kksStable} in the natural
way  to obtain  averaged limiting distributions for the 
random walk itself, but for the sake of space we do not state the
precise statement here.  
It should be noted that a formula for the scaling parameter $b>0$ appearing above when $\k<2$ has been obtained recently in \cite{eszStable,estzWQL}. 

It was not until more recently that the limiting distributions of the hitting time and the random walk were studied under the quenched distribution.
In the case when $\k>2$, Alili proved a quenched central limit theorem for the hitting times of the form
\be\label{qcltTn}
 \lim_{n\ra\infty} P_\w\left( \frac{T_n - E_\w T_n}{\s_1 \sqrt{n}} \leq x \right) = \Phi(x), \quad \forall x\in\R, \quad P-a.s.,
\ee
where $\s_1^2 = E_P[ \Var_\w T_1 ] < \infty$ \cite{aRWRE}. 
The environment-dependent centering term $E_\w T_n$ makes it difficult to use \eqref{qcltTn} to obtain a quenched central limit theorem for the random walk, but this difficulty was overcome independently by Goldsheid \cite{gQCLT} and Peterson \cite{pThesis} to obtain a quenched central limit theorem for the random walk (also with an environment-dependent centering). 

When $\k<2$ the situation is quite different. Even though one could
reasonably expect that, similarly to \eqref{qcltTn}, a limiting stable
distribution of index $\k$ existed (possibly with environment-dependent
centering or scaling), this has turn out not be the case. In fact, it
was shown in \cite{pzSL1,p1LSL2} 
that quenched limiting distributions do not exist when $\k<2$.  
For $P$-a.e. environment $\w$, there exist two (random) subsequences
$n_k = n_k(\w)$ and $m_k=m_k(\w)$ so that the limiting distributions
of $T_{n_k}$ and $T_{n_k'}$ under the measure $P_\w$ are Gaussian and
shifted exponential, respectively. That is, 
\[
 \lim_{k\ra\infty} P_\w\left( \frac{T_{n_k} - E_\w T_{n_k}}{\sqrt{\Var_\w T_{n_k}}} \leq x \right) = \Phi(x), \quad \forall x \in \R,
\]
and
\[
\lim_{k\ra\infty} P_\w\left( \frac{T_{m_k} - E_\w T_{m_k}}{\sqrt{\Var_\w T_{m_k}}} \leq x \right) = 
\begin{cases}
 1-e^{-x-1} & x > -1 \\
0 & x \leq -1,
\end{cases}
\quad \forall x \in \R.
\]
These subsequences were then used to show the non-existence of quenched limiting distributions for the random walk as well \cite{pzSL1,p1LSL2}. 

These results of \cite{pzSL1,p1LSL2} are less than completely
satisfying because one would like to be able to say something about
the quenched distribution after a large number of steps. Also, the
existence of subsequential limiting distributions that are Gaussian
and shifted exponential begs the question of whether and what other
types of distributions are possible to obtain through subsequences.  
The proof of the non-existence of quenched limiting distributions in
\cite{p1LSL2} implies, for large $n$, the magnitude of the hitting
time  $T_n$ is determined, to a large extent, by the amount of time it
takes the random walk to pass a few 
``large traps'' in the interval $[0,n]$. Moreover, as was shown in
\cite[Corollary 4.5]{p1LSL2}, the time to cross a ``large trap'' is
approximately an exponential random variable with parameter depending
on the ``size'' of the trap. Therefore, one would hope that the
quenched distribution of $T_n$ could be described in terms of some
random (depending on $\w$) weighted sum of exponential random variables.  
Our main results confirm this by showing that the quenched
distribution -- viewed as a random probability measure on $\R$ --
converges in distribution on the space of probability measures
to the law of a certain random infinite weighted sum  of exponential
random variables.

Before stating our main result, we introduce some notation. Let
$\mathcal{M}_1$ be the space of probability measures on
$(\R,\mathcal{B}(\R))$, where $\mathcal{B}(\R)$ is the Borel
$\s$-field.  
Recall that $\mathcal{M}_1$ is a complete, separable metric space when
equipped with the Prohorov metric 
\be\label{M1metric}
 \rho(\pi,\mu) = \inf \{ \e>0 \, : \, \pi(A) \leq \mu(A^\e ) + \e, \, \,
\mu(A) \leq \pi(A^\e ) + \e \, \,  \forall A\in \mathcal{B}(\R) \},
\quad \pi,\mu \in \mathcal{M}_1,  
\ee
where $A^\e := \{ x \in \R \, : \, |x-y| < \e \text{ for some } y \in
A \}$ is the $\e$-neighbourhood of $A$.  
By a random probability measure we mean a $\mathcal{M}_1$-valued
random variable, and we denote convergence in distribution of a
sequence of random probability measures by $\mu_n \Lra \mu$; see
\cite{bCOPM}. This
notation does carry the danger of being confused with the weak
convergence of probability measures on $\R$, but we prefer it to the
more proper, but awkward, notation ${\mathcal L}_{\mu_n}\Lra
{\mathcal L}_{\mu}$ with ${\mathcal L}_{\mu}$ being the law of a
random measure $\mu$. 

Next, let $\mathcal{M}_p$ be the space of Radon point processes on
$(0,\infty]$; these are the point processes assigning a finite mass to
 all  sets $(x,\infty]$ with $x>0$.  We equip $\mathcal{M}_p$
   with the standard topology of vague convergence. This topology
   can be metrized to make $\mathcal{M}_p$ a complete separable metric
   space; see \cite[Proposition 3.17]{rEVRVPP}. For point processes in
   $\mathcal{M}_p$ we denote  vague convergence by $\zeta_n
   \overset{v}{\ra}  \zeta$.   
An $\mathcal{M}_p$-valued random variable will be called a random point
process,  and, as above, we will use the somewhat improper notation
$\zeta_n \Lra \zeta$ to denote convergence in distribution of random
point processes.

We define a mapping $\bar{H}:\mathcal{M}_p \ra \mathcal{M}_1$ in the
following manner. Let $\zeta = \sum_{i\geq 1} \d_{x_i}$, where $(x_i)$
is an arbitrary enumeration of the points of $\zeta\in
\mathcal{M}_p$. We let $\bar H(\zeta)$ to be the probability measure
defined by   
\be\label{barHdef}
 \bar{H}(\zeta)(\cdot) = 
\begin{cases}
\Pv\left( \sum_{i\geq 1} x_i(\tau_i -1) \in \cdot \,  \right) & \sum_{i\geq 1} x_i^2 < \infty \\
\d_0(\cdot)& \text{otherwise}, 
\end{cases}
\ee
where, under a probability measure $\Pv$, 
$(\tau_i)$ is a sequence of i.i.d.\ mean 1 exponential random
variables. Note that the condition $\sum_{i\geq 1} x_i^2 < \infty$
guarantees that the sum inside the probability converges $\Pv$-a.s. It
is clear that the mapping $\bar H$ is well defined in the sense that
$\bar{H}(\zeta)$ does not depend on the enumeration of the points of
$\zeta$. 
We defer the proof of the following lemma to Appendix \ref{Hmeas}.  

\begin{lem} \label{l:meas}
The map $\bar{H}$ is measurable.
\end{lem}

We are now ready to state our first main result, describing the weak
quenched limiting distribution for the hitting times centered by the
quenched mean.  
\begin{thm}\label{wqlTn}
 Let Assumptions \ref{iidasm} - \ref{kasm} hold, and 
for any $\w \in \Omega$ let $\mu_{n,\w}\in \mathcal{M}_1$ be defined by 
\begin{equation} \label{e:mu.bar}
{\bar \mu}_{n,\w}(\cdot) = P_\w\left( \frac{T_n - E_\w T_n}{n^{1/\k}} \in \cdot \right).
\end{equation}
 Then  there exists a $\l>0$ such that
 $\bar{\mu}_{n,\w} \Lra \bar H(N_{\l,\k})$ where $N_{\l,\k}$ is a
 non-homogeneous Poisson point process on $(0,\infty)$ with intensity
 $ \l x^{-\k-1}$. 
\end{thm}
\begin{rem}
 The Gaussian and centered exponential distributions that were shown in \cite{p1LSL2} to be subsequential quenched limiting distributions of the hitting times are both, clearly, in the support of the random limiting probability measure
obtained in Theorem \ref{wqlTn}. 
Indeed, letting $\zeta_k = k \delta_{k^{-1/2}} \in \mathcal{M}_p$ we see that
$\bar{H}(\zeta_1)$ is a centered exponential distribution, and the central limit theorem implies that $\lim_{k\ra\infty} \bar{H}(\zeta_k)$ is a standard Gaussian distribution. 
\end{rem}
\begin{rem} \label{r:poisson.rep}
One can represent the non-homogeneous Poisson process $N_{\l,\k}$ as
$$
N_{\l,\k}
= \sum_{j=1}^\infty \delta_{ (\l/\k)^{1/\k} \Gamma_j^{-1/\k}}\,,
$$
where $(\Gamma_j)_{j\geq 1}$ is the increasing sequence of the points
of the unit rate homogeneous Poisson process on $(0,\infty)$. In
particular, the points of $N_{\l,\k}$ are
square summable with probability 1 if  $\k<2$ (and square summable
with probability 0 if  $\k\geq 2$.) Furthermore, the random limiting
distribution in Theorem \ref{wqlTn} can be written in the form 
\begin{equation} \label{e:repr.lim}
\bar H(N_{\l,\k})(\cdot) = \Pv\Bigl( (\l/\k)^{1/\k} \sum_{j=1}^\infty
\Gamma_j^{-1/\k} (\tau_j-1)\in\cdot\Bigr)\,,
\end{equation}
and we recall that the probability in \eqref{e:repr.lim} is taken with
respect to the exponential random variables $(\tau_j)$, while keeping
the standard Poisson arrivals $(\Gamma_j)$ fixed. 

The random probability measure $L=H(N_{\l,\k})$ above has a curious
stability property in $\mathcal{M}_1$: if $L_1,\ldots, L_n$ are
i.i.d. copies of $L$, then 
\be \label{e:stab}
L_1\ast \ldots \ast L_n (\cdot)\eil L\bigl(\cdot/n^{1/\k}\bigr)
\ee
for $n=1,2,\ldots$. To see why this is true, represent each $L_i$ as
in \eqref{e:repr.lim}, but using an independent sequence of Poisson
arrivals for each $i=1,\ldots, n$. Then the $n$-fold convolution
$L_1\ast \ldots \ast L_n$ has the same representation, but the
sequence of the  standard Poisson arrivals has to be replaced by a
superposition of $n$ such independent sequences. Since a superposition
of independent Poisson processes is, once again, a Poisson process and
the mean measures add up, we conclude that 
$$
L_1\ast \ldots \ast L_n (\cdot)\eil \Pv\Bigl( (\l/\k)^{1/\k} \sum_{j=1}^\infty
\tilde\Gamma_j^{-1/\k} (\tau_j-1)\in\cdot\Bigr)\,,
$$
where $(\tilde\Gamma_j)_j$ is the increasing sequence of the points
of a homogeneous Poisson random measure on $(0,\infty)$ with intensity
$n$. Since the sequence $(\Gamma_j/n)_j$ also forms a Poisson random
measure with intensity  $n$, \eqref{e:stab} follows. 
\end{rem}

Since we know that when $\kappa<2$ there is no centering and scaling that results in convergence to a deterministic distribution, 
we have some flexibility in choosing what
centering and scaling to work with. For example, if we use the averaged
centering and scaling in Theorem \ref{kksthm}, then a 
slightly different random probability distribution will appear in the 
limit. Before stating this result we need to introduce some more
notation. Define mappings  $H,H_\e:\mathcal{M}_p \ra
\mathcal{M}_1$, $\e>0$, as follows. For $\zeta = \sum_{i\geq 1} \d_{x_i}$, 
$H(\zeta)$ and $H_\e(\zeta)$ are the probability measures defined by 
\be\label{Hdef}
 H(\zeta)(\cdot) = 
\begin{cases}
 \Pv \left( \sum_{i\geq 1} x_i \tau_i \in \cdot \right) & \text{if } \sum_{i\geq 1} x_i < \infty \\
 \d_0 & \sum_{i\geq 1} x_i = \infty. 
\end{cases}
\ee
and 
\be\label{Hedef}
 H_\e(\zeta)(\cdot) = \Pv \left( \sum_{i\geq 1} x_i \tau_i \ind{x_i > \e} \in \cdot \right).
\ee
As was the case in the definition of $\bar{H}$ in \eqref{barHdef}, the
definition of $H(\zeta)$ does not depend on a particular enumeration
of the points of $\zeta$. Furthermore, an obvious modification of the
proof of Lemma \ref{l:meas} shows that the map $H$ is measurable. The
  maps $H_\e$ are even (almost) continuous, as will be seen in Section \ref{averagedcentering}.

\begin{thm}\label{wqlTnA}
Let Assumptions \ref{iidasm} - \ref{kasm} hold. For $\l,\k>0$ let
$N_{\l,\k}$  be a non-homogeneous Poisson point process on
$(0,\infty)$ with intensity $ \l x^{-\k-1}$. Then for every
$\k\in(0,2)$ there is a $\l>0$ such that the following statements hold. 
\begin{enumerate}
 \item If $\kappa \in(0,1)$, then 
\[
 \mu_{n,\w}(\cdot) = P_\w\left( \frac{T_n}{n^{1/\k}} \in \cdot \right) \Lra H(N_{\l,\k}).
\]
 \item If $\kappa = 1$, then 
\[
 \mu_{n,\w}(\cdot) = P_\w\left( \frac{T_n-n D(n)}{n} \in \cdot \right)
 \Lra \lim_{\e\ra 0^+} \bigl[ H_\e(N_{\l,1})*\d_{-c_{\l,1}(\e)}\bigr], 
\] 
where $c_{\l,1}(\e) =
\int_{\e}^1 \l x^{-1} \, dx= \l \log(1/\e)$, and $D(n)$ is a sequence such that $D(n) \sim A \log n$ for some $A>0$. 
 \item If $\kappa \in (1,2)$, then 
\[
 \mu_{n,\w}(\cdot) = P_\w\left( \frac{T_n - n/\vp}{n^{1/\kappa}} \in
 \cdot \right) \Lra \lim_{\e\ra 0^+} \bigl[
 H_\e(N_{\l,\k})*\d_{-c_{\l, \k}(\e)}\bigr], 
\]
where $c_{\l,\k}(\e) = \int_{\e}^\infty \l  x^{-\k} \, dx =
\frac{\l}{\k-1}  \e^{-(\k -1)}$.  
\end{enumerate}
\end{thm}
\begin{rem}
The limits as $\e\to 0^+$ in the cases $1\leq \k<2$ in Theorem
 \ref{wqlTnA} are weak limits in $\mathcal{M}_1$. The fact that these
 limits exist is standard; see
 e.g. \cite{samorodnitsky:taqqu:1994}. As we show in Section \ref{averagedcentering}, fixing a Poisson process
 $N_{\l,\k}$ on some probability space (for example, as in Remark
 \ref{r:poisson.rep}), even convergence with probability 1 holds.

The limiting random probability measures obtained in the different
parts of Theorem \ref{wqlTnA} also have stability properties in
$\mathcal{M}_1$, similar to the stability property of $\bar{H}(N_{\l,\k})$
described in Remark \ref{r:poisson.rep}. Specifically, if
$L_1,L_2,\ldots, L_n$ are i.i.d. copies of  the  
limiting random probability measure $L$ in Theorem \ref{wqlTnA}, then
the stability relation for the convolution operation \eqref{e:stab}
still holds if $\k\not=1$. In the case $\k=1$, the corresponding
stability relation is
\be\label{e:stab2}
L_1\ast \ldots \ast L_n (\cdot)\eil L\bigl(\cdot/n-\l\log n\bigr)\,.
\ee
The proof is similar to the argument used in Remark
\ref{r:poisson.rep}. We omit the details. 
\end{rem}
The statement (and proof) of the weak quenched limits with the
quenched centering (Theorem \ref{wqlTn}) is much simpler than the
corresponding result with the averaged centering (Theorem
\ref{wqlTnA}). However, in transferring a limiting distribution from
the hitting times $T_n$ to the location of the random walk $X_n$ it is
easier to use the averaged centering. 

\begin{cor}\label{wqlXn}
Let Assumptions \ref{iidasm} -- \ref{kasm} hold for some $\k\in(0,2)$,
and let $\l>0$ be given by Theorem \ref{wqlTnA}. 
\begin{enumerate}
 \item\label{CorCase1} If $\kappa \in(0,1)$, then for any $x\in\R$, 
\[
 P_\w\left( \frac{X_n}{n^{\k}} < x \right) \Lra
 H(N_{\l,\k})(x^{-1/\k},\infty). 
\]
\item\label{CorCase2} If $\kappa = 1$, then there exists a sequence $\d(n) \sim n/(A \log n)$ (with $A>0$ as in the conclusion of Theorem \ref{wqlTnA}) such that for any $x\in\R$,
\[
 P_\w\left( \frac{X_n-\d(n)}{n/(\log n)^2} < x \right) \Lra \lim_{\e\ra 0^+} \left(H_\e(N_{\l,1}) * \d_{-c_{\l,1}(\e)} \right) (-A^2 x, \infty).
\] 
 \item\label{CorCase3} If $\kappa \in (1,2)$, then for any $x\in\R$,
\[
 P_\w\left( \frac{X_n - n \vp}{n^{1/\kappa}} < x \right) \Lra
 \lim_{\e\ra 0^+} \left(H_\e(N_{\l,\k}) * \d_{-c_{\l,\k}(\e)} \right)(-x \vp^{-1-1/\k}, \infty).
\]
\end{enumerate}
\end{cor}
\begin{rem}
 The type of convergence in Corollary \ref{wqlXn} is weaker than that in Theorems \ref{wqlTn} and \ref{wqlTnA}. Instead of proving that the quenched distribution of $X_n$ (centered and scaled) converges in distribution on the space $\mathcal{M}_1$, we only prove that certain projections of the quenched law converge in distribution as real valued random variables. We suspect that, with some extra work, the techniques of this paper could be used to prove a limiting distribution for the full quenched distribution of $X_n$, but we will leave that for a future paper. Some results in this direction have previously been obtained in \cite{eszAging}
\end{rem}
\begin{rem}
 Theorem \ref{wqlTnA} and Corollary \ref{wqlXn} generalize the stable limiting distributions under the averaged law \cite{kksStable}. For instance, when $\kappa \in (0,1)$, 
\[
 \P\left( \frac{T_n}{n^{1/\k}} \leq x \right) = E_P\left[ P_\w \left( \frac{T_n}{n^{1/\k}} \leq x \right) \right] \limn \Ev [ H(N_{\l,\k})(-\infty,x]],
\]
and it is easy to see that $\Ev [H(N_{\l,\k})(-\infty,x]] = L_{\k,b}(x)$ for some $b>0$. 
\end{rem}

The structure of the paper is as follows. In Section \ref{bg} we introduce some notation and review some basic facts that we will need. Then, in Section \ref{gmsection} we outline a general method for transferring a limiting distribution result for one sequence of random probability measures to another sequence of random probability measures by constructing a coupling between the two sequences.
The method developed in Section \ref{gmsection} is then implemented several times in Section \ref{Transfer} to reduce the study of the quenched distribution of the hitting times $T_n$ to the quenched distribution of a certain environment-dependent mixture of exponential random variables. 
Then, these environment-dependent mixing coefficients are shown in Section \ref{betaanalysis} to be related to a non-homogeneous Poisson point process $N_{\l,\k}$. 
In Section \ref{quenchedcentering} we complete the proof of Theorem \ref{wqlTn} by proving a weak quenched limiting distribution for this mixture of exponentials. The proof of Theorem \ref{wqlTnA} is similar to the proof of Theorem \ref{wqlTn}, and in Section \ref{averagedcentering} we indicate how to complete the parts of the proof that are different. Finally, in Section \ref{Tn2Xn} we give the proof of the Corollary \ref{wqlXn}. 

Before turning to the proofs, we make one remark on the writing style. Throughout the paper, we will use $c$, $C$, and $C'$ to denote generic constants that may change from line to line. Specific constants that remain fixed throughout the paper are denoted $C_0$, $C_1$, etc.

\begin{rem}
 Soon after this work had been completed and posted on the arXiv, two other papers \cite{dgWQL,estzWQL} appeared giving independent proofs of some of the main results of this paper. A few brief remarks are in order on the differences between these papers. 
Neither of the above papers state weak quenched limits with the averaged centering as in Theorem \ref{wqlTnA} (although this should follow easily from Corollary 1 in \cite{dgWQL}) nor do they discuss the quenched distribution of $X_n$ as in Corollary \ref{wqlXn}. 
In \cite{dgWQL}, instead of studying the hitting times $T_n$ directly the authors study the amount of time spent in the interval $[0,n)$ - which can easily be seen to have the same weak quenched limiting distributions as the hitting times. In \cite{estzWQL}, the authors prove Theorem \ref{wqlTn} under the stronger Wasserstein $W^1$ metric on the space $\mathcal{M}_1$. However, an analysis of the proof in the current paper (especially the coupling technique introduced in Section \ref{gmsection}) reveals that it should be easily adaptable to the Wasserstein metric as well. 
\end{rem}

\begin{rem}
 After the initial submission of this paper, we also became aware of \cite{stRPD} which provides a systematic study of stable random probability distributions - that is, random probability distributions with stability properties like \eqref{e:stab} or \eqref{e:stab2}. Moreover, in \cite{stRPD} the authors study a simpler model of random motion in a random environment and obtain weak quenched limiting distributions for the hitting times similar to Theorems \ref{wqlTn} and \ref{wqlTnA}. 
\end{rem}


\section{Background}\label{bg}
In this section we introduce some notation that will be used throughout the rest of the paper. 
For RWRE on $\Z$, many quenched probabilities and expectations are
explicitly solvable in terms of the environment. It is in order to
express these formulas compactly that we need this additional notation. 
Recall that $\rho_x = (1-\w_x)/\w_x$, $x\in\Z$.  Then, for $i\leq j$
we let 
\be\label{pirwdef}
 \Pi_{i,j} = \prod_{x=i}^j \rho_x, \quad R_{i,j} = \sum_{k=i}^j \Pi_{i,k}, \quad\text{and}\quad W_{i,j} = \sum_{k=i}^j \Pi_{k,j}. 
\ee
Denote  
\be\label{rwdef}
 R_i = \lim_{j\ra\infty} R_{i,j} = \sum_{k=i}^\infty \Pi_{i,k} \quad\text{and}\quad W_j = \lim_{i\ra-\infty} W_{i,j} = \sum_{k=-\infty}^j \Pi_{k,j}. 
\ee
Note that Assumption \ref{tasm} implies that $R_i$ and $W_j$ are
finite with probability 1 for all $i,j\in\Z$. 
The following formulas are extremely useful (see \cite{zRWRE} for a reference)
\be\label{HPform}
 P_\w^x \left( T_i > T_j \right) = \frac{R_{i,x-1}}{R_{i,j-1}} \quad
 \text{and} \quad  P_\w^x \left( T_i < T_j \right) =
 \frac{\Pi_{i,x-1}R_{x,j-1}}{R_{i,j-1}}, \ \ i < x <
 j,   
\ee
\be\label{QETform}
 E_\w^{i} T_{i+1} = 1 + 2 W_i, \ \  i\in \Z. 
\ee

As in \cite{pzSL1,p1LSL2}, we define the ``ladder locations''
$\nu_i$ of the environment by 
\begin{align}
\nu_0 = 0, \quad\text{and}\quad \nu_i =
\inf\{n > \nu_{i-1} : \,  \Pi_{\nu_{i-1},n-1} < 1\}, \quad  i \geq 1.
\label{nudef}
\end{align}
Since the environment is i.i.d., the sections of the environment
$\{\w_x : \, \nu_{i-1} \leq x < \nu_i \}$ between successive ladder
locations are also i.i.d.  However, the environment directly to the
left of $\nu_0 = 0$ is different from the environment to the left of
$\nu_i$ for $i>1$. Thus, as in \cite{pzSL1,p1LSL2} it is convenient to
define a new probability law on environments by
\be\label{Qdef}
 Q(\cdot) = P\left( \cdot\, | \, \Pi_{i,-1}< 1, \text{ all } i\leq -1
 \right);  
\ee
by Assumption \ref{tasm} the condition is an event of positive
probability. 

Two facts about the distribution $Q$ will be important to keep in mind throughout the remainder of the paper. 
\begin{itemize}
 \item Under the measure $Q$ the environments stationary under shifts by the ladder locations $\nu_i$. 
 \item Since, under $P$, the environment is i.i.d., the measure $Q$ 
   coincides with the measure $P$ on $\s(\w_x : \, x\geq 0)$. 
\end{itemize}

Often for convenience we will denote $\nu_1$ by $\nu$. 
It was shown in \cite[Lemma 2.1]{pzSL1} that the distribution of $\nu$ (which is the same under $P$ and $Q$) has exponential tails. That is, there exist constants $C,C'>0$ such that 
\be\label{nutail}
 P(\nu > x ) = Q(\nu > x) \leq C' e^{-C x}, \  x\geq 0. 
\ee
In particular this implies that $\lim_{n\ra\infty} \nu_n/n = \bar\nu
:= E_Q\nu = E_P\nu$, both $P$ and $Q$ - a.s.. 

In contrast, it was shown in \cite[Theorem 1.4]{pzSL1} that, under
Assumption \ref{kasm}, the distribution of the first hitting time 
 $E_\w T_\nu$ has power tails under the measure $Q$. That is, there
exists a constant $\tc$ such that  
\be\label{ETnutail}
 Q( E_\w T_\nu > x ) \sim \tc x^{-\k}, \  x\ra\infty. 
\ee

\section{A General Method for Transferring Weak Quenched Limits}\label{gmsection}

Our strategy for proving weak quenched limits for the hitting times
will be to first prove a weak quenched limiting distribution for a
related sequence of random variables.  
Then by exhibiting a coupling between the two sequences of random
variables we will be able to conclude that the hitting times have the
same weak quenched limiting distribution. 
The second of these steps is accomplished through the following
lemma. It applies to random probability measures on $\R^2$, which are
simply random variables taking values in $\mathcal{M}_1(\R^2)$. The
latter space is the space of all probability measures on $\R^2$ which
can be turned into a complete, separable metric space in the same way
as it was done to the space $\mathcal{M}_1$ in Section
\ref{sec:intro}. The two maps assigning each probability measure
in $\mathcal{M}_1(\R^2)$ its two marginal probability measures are
automatically continuous. 

\begin{lem}\label{GeneralMethod}
Let $\theta_n, \, n=1,2,\ldots$ be a sequence of random probability
measures on $\R^2$ defined on some probability space $\ProbSpace$. Let
$\gamma_{n}$ and $\gamma_{n}'$ be the two marginals of $\theta_n$,
$n=1,2,\ldots$. Suppose that for every $\delta>0$
\be\label{probcouple}
 \lim_{n\ra\infty} \Pv \left( \theta_n\bigl( \bigl\{ (x,y):\,
 |x-y|\geq\d\bigr\}\bigr) >\d\right) = 0. 
\ee
If $\gamma_{n} \Lra \gamma$ for some $\gamma\in \mathcal{M}_1$,
then $\gamma_{n}' \Lra \gamma$ as well. 
\end{lem}
\begin{rem}
 Generally the space $\Omega$ will be the space of environments and
 $\Pv$ will be the measure $Q$ on environments  defined in
 \eqref{Qdef}. However, in one application (Lemma \ref{P2Q} below) we
 will use slightly different spaces and measures and so we need to
 state Lemma \ref{GeneralMethod} in this more general form.  
\end{rem}

\begin{proof}
The definition of the Prohorov metric $\rho$ in \eqref{M1metric}
implies that, if $\theta_n\bigl( \bigl\{ (x,y):\, 
 |x-y|\geq\d\bigr\}\bigr) \leq \d$, then
 $\rho(\gamma_{n},\gamma_{n}') \leq \d$. Therefore, the assumption
 \eqref{probcouple} implies that $\rho(\gamma_{n},\gamma_{n}')\to 0$
 in probability. Now the statement of the lemma follows from Theorem
 3.1 in \cite{bCOPM}. 
\end{proof}

The following is an immediate corollary. 
\begin{cor}\label{GeneralMethodcor}
 Under the setup of Lemma \ref{GeneralMethod}, assume that 
\be\label{meancouple}
\Ev_{\theta_n}|X-Y| \ra 0, \quad\text{ in $\Pv$-probability}
\ee
(here $X$ and $Y$ are the coordinate variables in $\R^2$ and $\Ev_{\theta_n}$ is expectation with respect to the measure $\theta_n$). 
If $\gamma_{n} \Lra \gamma$ for some $\gamma\in \mathcal{M}_1$,
then $\gamma_{n}' \Lra \gamma$ as well. 
\end{cor}
\begin{proof}
The claim follows immediately from Lemma \ref{GeneralMethod} and
Markov's inequality via 
\[
\Pv \left(  \theta_n (|X - Y| \geq \d) \geq \d \right) \leq  \Pv(
\Ev_{\theta_n} |X - Y| \geq \d^2 ). 
\]
\end{proof}

\begin{rem}\label{vcrem}
By the Cauchy-Scwarz inequality, a sufficient condition for
\eqref{meancouple} is 
\be\label{varcouple}
\Ev_{\theta_n} (X - Y ) \ra 0 \quad\text{and}\quad \Varv_{\theta_n} (
X - Y ) \ra 0, \quad\text{ in $\Pv$-probability}. 
\ee
\end{rem}

\section{A Series of Reductions}\label{Transfer}
In this section we repeatedly apply Lemma \ref{GeneralMethod} and
Corollary \ref{GeneralMethodcor} to reduce the problem of finding weak
quenched limits of the hitting times $T_n$ to the problem of finding
weak quenched limits of a simpler sequence of random variables that is
a random mixture of exponential distributions.  

First of all, instead of studying the quenched distributions of the
hitting times, it will be more convenient to study the hitting times
along the random sequence of the ladder locations $\nu_n$. Since by
\eqref{nutail}, the distance between consecutive ladder locations has
exponential tails, and $\nu_n/n \ra \bar\nu = E_P \nu_1$ the quenched
distribution of $T_n$ should be close to the quenched distribution of
$T_{\nu_{\nc n}}$ with $\nc = 1/\bar\nu$ (for ease of notation we will
write $\nu_{\nc n}$ instead of $\nu_{\fl{\nc n}}$). Based on this, we will
reduce our problem to proving a quenched weak limit theorem for
$T_{\nu_n} = \sum_{i=1}^n (T_{\nu_i} - T_{\nu_{i-1}})$.  Secondly, as
mentioned in the introduction, the proof of the non-existence of
quenched limiting distributions for hitting times in \cite{p1LSL2}
hinged on two observations. The first of these says that, for large
$n$, the magnitude of $T_{\nu_n}$ is mainly determined by the
increments $T_{\nu_i} - T_{\nu_{i-1}}$ for those $i=1,\ldots, n$ for
which there is a large ``trap'' between the ladder
locations $\nu_{i-1}$ and $\nu_i$. The second observation is that,
when there is a 
large ``trap'' between $\nu_{i-1}$ and $\nu_i$, the time to cross from
$\nu_{i-1}$ to $\nu_i$ is, approximately, an exponential random
variable with a large mean. That is, $T_{\nu_i} - T_{\nu_{i-1}}$ may
be approximated by $\b_i \tau_i$ where  
\be\label{bdef}
 \b_i = \b_i(\w) = E_\w^{\nu_{i-1}} T_{\nu_i} = E_\w( T_{\nu_i} - T_{\nu_{i-1}} ), 
\ee
and $\tau_i$ is a mean 1 exponential random variable that is
independent of everything else. 

When analyzing the hitting times of the ladder locations $T_{\nu_n}$
the measure $Q$ is more convenient to use than the measure $P$ since,
under $Q$, the environment is stationary under shifts of the
environment by the ladder locations. In particular, $\{ \b_i \}_{i\geq
  1}$ is a stationary sequence under $Q$.  
The main result of this
section is the following proposition. 

\begin{prop}\label{bigtransfer}
 For $\w \in \Omega$, suppose that $P_\w$ is expanded so that there
 exists a sequence $\tau_i$ which, under $P_\w$, is an 
 i.i.d.\ sequence of mean 1 exponential random variables. Let
 $\bar{\s}_{n,\w} \in \mathcal{M}_1$ be defined by 
\be\label{sdef}
 \bar{\s}_{n,\w}(\cdot) = P_\w\left( \frac{1}{n^{1/\k}} \sum_{i=1}^n \b_i (\tau_i - 1) \in \cdot \right),
\ee
where $\b_i=\b_i(\w)$ is given by \eqref{bdef}. 
If $\bar{\s}_{n,\w} \overset{Q}{\Lra} \bar{H}(N_{\l,\k})$ then
$\bar{\mu}_{n,\w} \overset{P}{\Lra} \bar{H}(N_{\l/\bar{\nu},\k})$,
where ${\bar \mu}_{n,\w}$ is defined in \eqref{e:mu.bar}. 
\end{prop}

Lemma \ref{GeneralMethod} says that weak imits for one sequence of
$\mathcal{M}_1$-valued random variables can be transferred to another
sequence of $\mathcal{M}_1$-valued random variables if these random
probability measures can be coupled in a nice way. We pursue this idea
and prove Proposition \ref{bigtransfer} by establishing the series of
lemmas below. All of these results will be proved using Lemma
\ref{GeneralMethod} and Corollary \ref{GeneralMethodcor}.  

\begin{lem}\label{P2Q}
 If $\bar{\mu}_{n,\w} \overset{Q}{\Lra} \bar{H}(N_{\l,\k})$ then $\bar{\mu}_{n,\w} \overset{P}{\Lra} \bar{H}(N_{\l,\k})$. 
\end{lem}

\begin{lem}\label{Tn2Tnun}
 For $\w \in \Omega$, let $\bar\phi_{n,\w} \in \mathcal{M}_1$ be defined by 
\[
 \bar\phi_{n,\w}(\cdot) = P_\w\left( \frac{T_{\nu_n} - E_\w T_{\nu_n}}{n^{1/\k}} \in \cdot \right) = P_\w\left( \frac{1}{n^{1/\k}} \sum_{i=1}^n ( T_{\nu_i} - T_{\nu_{i-1}} - \b_i ) \in \cdot \right).
\]
If $\bar{\phi}_{n,\w} \overset{Q}{\Lra} \bar{H}(N_{\l,\k})$ then $\bar{\mu}_{n,\w} \overset{Q}{\Lra} \bar{H}(N_{\l/\bar{\nu},\k})$.
\end{lem}

\begin{lem}\label{Tnun2bt}
 If $\bar{\s}_{n,\w} \overset{Q}{\Lra} \bar{H}(N_{\l,\k})$ then $\bar{\phi}_{n,\w} \overset{Q}{\Lra} \bar{H}(N_{\l,\k})$.
\end{lem}

\begin{proof}[Proof of Lemma \ref{P2Q}] 
 Recall that $P$ and $Q$ are identical on $\s(\w_x : \, x\geq 0)$. We
 start with a coupling of $P$ and $Q$ that that produces two
 environments that agree on
 the non-negative integers. Let $\w$ be an environment with
 distribution $P$ and let $\tilde{\w}$ be an independent environment
 with distribution $Q$. Then, construct the environment $\w'$ by
 letting  
\[
 \w'_x = \begin{cases} \tilde{\w}_x & x \leq -1 \\ \w_x & x \geq 0. \end{cases}
\]
Then $\w'$ has distribution $Q$ and is identical to $\w$ on the
non-negative integers. Let $\Pv$ be the joint distribution of
$(\w,\w')$ in the above coupling. 
Given a pair of environments $(\w,\w')$, we will construct coupled 
random walks $\{X_n\}$ and $\{X_n'\}$ with hitting times $\{T_n\}$ and
$\{T_n'\}$, respectively, so that the marginal distributions of
$\{X_n\}$ and $\{X_n'\}$ 
are $P_\w$ and $P_{\w'}$ respectively. Let $P_{\w,\w'}$ denote the
joint distribution of $\{X_n\}$ and $\{X_n'\}$ with expectations denoted by
$E_{\w,\w'}$, and consider  random probability measures on $\R^2$
defined by
$$
\theta_n(\cdot) = P_{\w,\w'}\left[ \left( \frac{T_n -
    E_{\w,\w'}T_n}{n^{1/\k}}, \frac{T_n' -
    E_{\w,\w'}T_n'}{n^{1/\k}}\right)\in \cdot\right].
$$
We wish to construct the coupled random walks so that 
\be\label{wwTn}
 \lim_{n\ra\infty} n^{-1/\k} E_{\w,\w'}|(T_n - E_{\w,\w'}T_n) - (T_n'
 - E_{\w,\w'}T_n')| = 0, \quad \Pv-a.s. 
\ee
This will be more than enough to satisfy conditions \eqref{meancouple}
of Corollary \ref{GeneralMethodcor}, and the conclusion of Lemma
\ref{P2Q} will follow.  

We now show how to construct coupled random walks  $\{X_n\}$ and
$\{X_n'\}$. Since the environments $\w$ and $\w'$ agree on the
non-negative integers, our coupling will cause the two walks to move
in the same manner at all locations $x\geq 0$.   
Precisely, on their respective $i$th visits to site $x\geq 0$, they
will both either move to the right or both move to the left.  
To do this, let $\bar{\xi}=\{\xi_{x,i}\}_{x\in\Z, i\geq 1}$ be a
collection of i.i.d.\ standard uniform random variables that is
independent of everything else. Then, given $(\w,\w')$ and
$\bar{\xi}$, 
construct the random walks as follows: 
\[
 X_0 = 0, \quad\text{and}\quad
X_{n+1} = 
 \begin{cases}
  X_n + 1 & \text{if } X_n = x, \, \#\{ k\leq n: X_k = x \} = i, \, \text{and } \xi_{x,i} \leq \w_x \\
  X_n - 1 & \text{if } X_n = x, \, \#\{ k\leq n: X_k = x \} = i, \, \text{and } \xi_{x,i} > \w_x
 \end{cases}
\]
and
\[
 X_0' = 0, \quad\text{and}\quad
X_{n+1}' = 
 \begin{cases}
  X_n' + 1 & \text{if } X_n' = x, \, \#\{ k\leq n: X_k' = x \} = i, \, \text{and } \xi_{x,i} \leq \w_x' \\
  X_n' - 1 & \text{if } X_n' = x, \, \#\{ k\leq n: X_k' = x \} = i, \, \text{and } \xi_{x,i} > \w_x'.
 \end{cases}
\]

Having constructed our coupling, we now turn to the proof of \eqref{wwTn}. It is enough in fact to show that
\be\label{wwTn2}
 \sup_n E_{\w,\w'}|T_n-T_n'| < \infty, \quad\text{and}\quad \sup_n
 |E_{\w,\w'}T_n - E_{\w,\w'}T_n'|<\infty, \quad \Pv \text{-a.s.}
\ee
To show the second inequality  in \eqref{wwTn2}, we use the explicit
formula \eqref{QETform} for the quenched expectations of hitting
times, so that  
\[
 E_\w T_n = n + 2 \sum_{i=0}^n W_i = n + 2 \sum_{i=0}^n (W_{0,i} +
 \Pi_{0,i}W_{-1}) = n + 2 \sum_{i=0}^n W_{0,i} + 2 W_{-1} R_{0,n}.  
\]
Similarly, (with the obvious notation for corresponding random
variables corresponding to $\w'$) 
\[
 E_{\w'} T_n' = n + 2 \sum_{i=0}^n W_{0,i}' + 2 W_{-1}' R_{0,n-1}' = n
 + 2 \sum_{i=0}^n W_{0,i} + 2 W_{-1}' R_{0,n}, 
\]
where the second equality is valid because $\w_x = \w_x'$ for all
$x\geq 0$. Thus,  
\[
 \sup_{n} |E_{\w,\w'}T_n - E_{\w,\w'}T_n'| = \sup_{n} 2 R_{0,n}
 |W_{-1} - W_{-1}'| = 2 R_0 |W_{-1} - W_{-1}'| < \infty, \quad
 \Pv\text{-a.s.} 
\]
Turning to the first inequality in \eqref{wwTn2}, let
\[
 L_n := \sum_{k=0}^{T_n} \ind{X_k < 0}, \qquad L_n' :=\sum_{k=0}^{T_n'} \ind{X_k < 0},
\]
be  the number of visits by by the walks $\{X_n\}$ and $\{X_n'\}$,
correspondingly, to the negative integers, by the time they reach site $x=n$. The coupling of $T_n$ and
$T_n'$ constructed above is such that $|T_n - T_n'| = |L_n-L_n'|$.  
Therefore,
\[
 E_{\w,\w'} |T_n-T_n'| = E_{\w,\w'} |L_n-L_n'| \leq E_{\w} L_n + E_{\w'} L_n'. 
\]
Letting $L = \lim_{n\ra\infty} L_n$ and $L' = \lim_{n\ra\infty} L_n'$
denote the total amount of time spent in the negative integers by the
random walks $\{X_n\}$ and $\{X_n'\}$, respectively, we need only to show
that $E_{\w} L + E_{\w'}L' < \infty$, $\Pv$-a.s. 
To this end, note that $L = \sum_{i=1}^G U_i$ where $G$ is the number
of times the random walk $\{X_n\}$ steps from $0$ to $-1$ and the
$U_i$ is the amount of time it takes to reach $0$ after the $i$th
visit to $-1$. Note that $G$ is a geometric random variable starting
from $0$ with success parameter $P_\w(T_{-1} = \infty) > 0$, and that
the $U_i$ are independent (and independent of $G$) with common
distribution equal to that of the time it takes a random walk in
environment $\w$ to reach 0 when starting at $-1$. 
Thus, by first conditioning on $G$, we obtain that 
\[
 E_{\w} L 
= E_\w\left[ G \left(E_\w^{-1} T_0\right) \right] = \left(E_\w^{-1} T_0\right)\frac{P_\w(T_{-1} < \infty)}{P_\w(T_{-1} = \infty)}.
\]
Similarly, 
\[
 E_{\w'} L' = \left(E_{\w'}^{-1} T_0\right)\frac{P_{\w'}(T_{-1} <
   \infty)}{P_{\w'}(T_{-1} = \infty)}. 
\]
This completes the proof since $E_\w^{-1} T_0$ and $E_{\w'}^{-1} T_0$
are finite, $\Pv$-a.s. by \eqref{QETform}. 
\end{proof}

\begin{proof}[Proof of Lemma \ref{Tn2Tnun}]
For $\w\in \Omega$, let $\hat{\phi}_{n,\w}\in\mathcal{M}_1$ be defined by
\[
 \hat{\phi}_{n,\w}(A) = P_\w\left( \frac{T_{\nu_{\nc n}} - E_\w T_{\nu_{\nc n}} }{n^{1/\k}} \in A \right) = \bar{\phi}_{\fl{\nc n},\w}\left( \frac{n^{1/\k}}{\fl{\nc n}^{1/\k}} \, A \right).  
\]
Since $n^{1/\k}/\fl{\nc n}^{1/\k} \ra \nc^{-1/\k} = \bar{\nu}^{1/\k}$ as $n\ra\infty$,  it
follows (for example, by Lemma \ref{GeneralMethod}) that 
\[
 \bar\phi_{n,\w}(\cdot) \overset{Q}{\Lra} \bar{H}(N_{\l,\k})(\cdot) \quad\text{implies that}\quad \hat\phi_{n,\w}(\cdot) \overset{Q}{\Lra} \bar{H}(N_{\l,\k})(\bar{\nu}^{1/\k} \, \cdot)
\]
Now, it follows from \eqref{e:repr.lim} that
$\bar{H}(N_{\l,\k})(\bar{\nu}^{1/\k} \, \cdot) \overset{Law}{=}
\bar{H}(N_{\l/\bar\nu,\k})(\cdot)$. Therefore, the claim of the lemma
will follow once we check that  
\be\label{Tnucn2Tn}
  \hat\phi_{n,\w} \overset{Q}{\Lra} \bar{H}(N_{\l,\k}) \quad\text{implies that}\quad \bar\mu_{n,\w} \overset{Q}{\Lra} \bar{H}(N_{\l,\k})
\ee

To show \eqref{Tnucn2Tn} we will verify condition \eqref{varcouple} of
the remark following Corollary \ref{GeneralMethodcor}.  Since both $\hat\phi_{n,\w}$ and
$\bar{\mu}_{n,\w}$ are mean zero distributions on $\R$, it is enough
to  show that
\be\label{vTnTnucn}
\lim_{n\ra\infty} Q\left( n^{-2/\k} \Var_\w(T_n - T_{\nu_{\nc n}}) > \d \right) = 0, \quad \forall \d>0. 
\ee
 To this end, note that if $\nu_{\nc n} \leq n \leq \nu_{k}$ then
 $\Var_\w(T_n - T_{\nu_{\nc n}}) = \sum_{x=\nu_{\nc n}+1}^{n}
 \Var_\w(T_x-T_{x-1}) \leq \Var_\w(T_{\nu_k} - T_{\nu_{\nc n}})$. A
 similar inequality holds if $\nu_k \leq n \leq \nu_{\nc n}$. Using this,
 we obtain that for any $\e>0$ 
\begin{align}
 Q\left( \Var_\w(T_n - T_{\nu_{\nc n}}) > \d n^{2/\k} \right) 
& \leq Q(|n-\nu_{\nc n}| > \e n ) + Q\left( \Var_\w(T_{\nu_{[\nc n]+[\e n]}} - T_{\nu_{\nc n}} ) > \d n^{2/\k} \right) \nonumber \\
&\qquad + Q\left( \Var_\w( T_{\nu_{\nc n}} - T_{\nu_{[\nc n]-[\e n]}}) > \d n^{2/\k} \right) \nonumber \\
& = Q(|n-\nu_{\nc n}| > \e n ) + 2 Q\left( \Var_\w(T_{\nu_{\e n}} ) > \d
 n^{2/\k} \right), \label{svar1} 
\end{align}
where the last equality is due to the fact that, under the measure
$Q$,  the environment is stationary under shifts of the ladder
locations. The first term in \eqref{svar1} vanishes since $\nu_{\nc n}/n
\ra 1$, $Q$-a.s., by the law of large numbers. For the second term in
\eqref{svar1}, recall that $n^{-2/\k} \Var_\w T_{\nu_n}$ has a
$\k$-stable limiting distribution under $Q$ \cite[Theorem
  1.3]{p1LSL2}. Thus, there exists a $b>0$ such that   
\[
 \lim_{n\ra\infty} Q\left( \Var_\w(T_{\nu_{\e n}} ) > \d n^{2/\k} \right) = 1 - L_{\k,b}(\d \e^{-2/\k}). 
\]
Since the right hand side can be made arbitrarily small by taking
$\e\ra 0$, we have finished the proof of \eqref{vTnTnucn} and, thus, 
also of the lemma.  
\end{proof}

\begin{proof}[Proof of Lemma \ref{Tnun2bt}]
The proof of the lemma consists of showing that we can couple the
standard exponential random variables of Proposition \ref{bigtransfer}
with the random walk $\{X_n\}$ in such a way that condition
\eqref{varcouple} of the remark following Corollary \ref{GeneralMethodcor} holds. Since the
relevant random probability measures have zero means, we only need to
ensure that  
\be\label{CheckVarCouple}
 \lim_{n\ra\infty} Q\left( n^{-2/\k} \Var_\w\left(T_{\nu_n} - E_\w T_{\nu_n} - \sum_{i=1}^n \beta_i(\tau_i-1) \right) > \d \right) 
= 0, \quad \forall \d>0.
\ee
We will perform the coupling in such a way that the sequence of pairs
$(T_{\nu_i}-T_{\nu_{i-1}}, \tau_i)$ is independent under the quenched
law $P_\w$. Since $E_\w T_{\nu_n} = \sum_{i=1}^n \b_i$, this will
imply that 
\[
 \Var_\w\left(T_{\nu_n} - E_\w T_{\nu_n} - \sum_{i=1}^n \beta_i(\tau_i-1) \right) 
= \sum_{i=1}^n \Var_\w\left( T_{\nu_i} - T_{\nu_{i-1}} - \beta_i \tau_i \right). 
\]
As in \cite{pzSL1}, for any $i$ define 
\be\label{Mdef}
 M_i = \max \{ \Pi_{\nu_{i-1},j}: \, \nu_{i-1}\leq j < \nu_i \}. 
\ee
The utility of the sequence $M_i$ is that it is roughly comparable to $\b_i$ and $\sqrt{\Var_\w(T_{\nu_i}-T_{\nu_{i-1}})}$, but $M_i$ is an i.i.d.\ sequence of random variables (see \cite[equations (15) and (63)]{pzSL1} for precise statements regarding these comparisons). 
In \cite[Lemma 5.5]{pzSL1} it was shown that for any $0<\e<1$, 
\[
 \lim_{n\ra\infty} Q\left( \frac{1}{n^{2/k}} \sum_{i=1}^n \Var_\w (T_{\nu_i}-T_{\nu_{i-1}})\ind{M_i \leq n^{(1-\e)/\k}} > \d\right) = 0, \quad \forall \d>0.  
\]
A similar argument (see also the proof of \cite[Lemma 3.1]{pzSL1}) implies that 
\[
 \lim_{n\ra\infty} Q\left( \frac{1}{n^{2/k}} \sum_{i=1}^n \b_i^2 \ind{M_i \leq n^{(1-\e)/\k}} > \d \right) = 0, \quad \forall \d>0. 
\]
Then, since $\Var_{\w}(T_{\nu_i} - T_{\nu_{i-1}} - \b_i \tau_i) \leq 2
\Var_{\w}(T_{\nu_i} - T_{\nu_{i-1}}) + 2 \b_i^2$,  in order to
guarantee \eqref{CheckVarCouple} it is enough to perform a  coupling
in such a way that for some $0<\e<1$, 
\be \label{e:req.couple}
 \lim_{n\ra\infty} Q\left( \frac{1}{n^{2/k}} \sum_{i=1}^n \Var_{\w}(T_{\nu_i} - T_{\nu_{i-1}} - \b_i \tau_i) \ind{M_i > n^{(1-\e)/\k}} > \d \right) = 0, \quad \forall \d>0.
\ee
Recall that we separately couple each exponential random variable
$\tau_i$ with the corresponding crossing time $T_{\nu_i}-T_{\nu_{i-1}}$. 
For simplicity of notation we will describe this coupling when $i=1$,
and we will denote $\nu_1$, $\beta_1$ and $\tau_1$ by $\nu$, $\beta$
and $\tau$, respectively.  

First, note that $T_{\nu}$ can be constructed by doing repeated excursions from the origin. Let $T_0^+ = \inf\{n > 0 :\, X_n = 0\}$ be the first return time to the origin, and let $\{F^{(j)}\}_{j\geq 1}$ be an i.i.d.\ sequence of random variables all having the distribution of $T_0^+$ under $P_\w(\cdot \, | \, T_0^+ < T_\nu )$. Also, let
let $S$ be independent of the $\{F^{(j)}\}$ and have the same
distribution as $T_\nu$ under $P_\w( \cdot \, | \, T_\nu < T_0^+ )$.  
Finally, let $N$ be independent of $S$ and the $\{F^{(j)}\}$ and have
a geometric distribution starting from $0$ with success parameter
$p_\w = P_\w(T_\nu < T_0^+ )$. Then we can construct $T_\nu$ by
letting  
\be\label{couplingsum}
 T_\nu = S + \sum_{j=1}^N F^{(j)}. 
\ee
Note that 
\be\label{beq}
 \b = E_\w T_\nu = E_\w S + \frac{1-p_\w}{p_\w}(E_\w F^{(1)})
\ee 
Given this construction of $T_\nu$, the most natural way to couple
$T_\nu$ with $\tau$ is to provide a coupling between $\tau$ and $N$.  
We set 
\be\label{couplingN}
 N= \fl{ c_\w \tau}, \qquad\text{where } c_\w = \frac{-1}{\log(1-p_\w)},
\ee
so that $N$ is exactly a geometric random variable with parameter
$p_\w$.  

For this coupling, we obtain the following bound on $\Var_\w(T_\nu - \b \tau)$.
\begin{lem}\label{varcouplelem}
 Let $T_\nu$ and $\beta \tau$ be coupled using \eqref{couplingsum} and \eqref{couplingN}. Then,
\be\label{varcoupleineq}
 \Var_\w(T_\nu - \b \tau) \leq (E_\w S)^2 + \frac{(E_\w F^{(1)})^2}{3} + \Var_\w(T_\nu) - (E_\w F^{(1)})^2 \Var_\w (N).
\ee
\end{lem}
\begin{proof}
First of all, note that
\begin{align}
& \Var_\w(T_\nu - \b \tau) 
= \Var_\w\left( S + \sum_{j=1}^N F^{(j)} - \b \tau \right) \nonumber \\
&\quad= \Var_\w( S) + \Var_\w\left( \sum_{j=1}^N F^{(j)} - \b \tau \right)  \nonumber \\
&\quad= \Var_\w( S)  + \Var_\w( F^{(1)} ) \left( E_\w \fl{c_\w \tau} \right) + \Var_\w\left(\fl{c_\w \tau} (E_\w F^{(1)}) - \b \tau \right) \label{couple3var}.
\end{align}
Since $\fl{c_\w \tau}$ is independent of $c_\w \tau-\fl{c_\w \tau}$, we
can use the identity for $\b$ in \eqref{beq} to write, with the help
of a bit of algebra,  
\begin{align*}
 \Var_\w\left(\fl{c_\w \tau} (E_\w F^{(1)}) - \b \tau \right) 
&= (E_\w F^{(1)})^2 \Var_\w\left(\fl{c_\w \tau} \right)  + \b^2 - 2 (E_\w F^{(1)})\b \Cov\left( \fl{c_\w \tau} , \tau \right) \\
&= \Bigl( (E_\w F^{(1)})^2 -2(E_\w F^{(1)})\b/c_\w\Bigr)
 \Var_\w\left(\fl{c_\w \tau}\right) + \b^2 \\
&= ( E_\w S)^2 + 2(E_\w S)(E_\w F^{(1)}) \frac{1-p_\w}{p_\w^2} \left( p_\w + \log(1-p_\w) \right) \\
& \qquad + (E_\w F^{(1)})^2 \frac{1-p_\w}{p_\w^2} \left(2-p_\w + 2 \frac{1-p_\w}{p_\w} \log(1-p_\w) \right).  
\end{align*}
Using a Taylor series expansion of $\log(1-p)$ for $|p|<1$, one can show that for any $p\in[0,1)$,
\[
 p+\log(1-p) = - \sum_{k=2}^{\infty} \frac{p^k}{k} \leq 0,
\]
and
\[
 \frac{1-p}{p^2}\left( 2-p + 2 \frac{1-p}{p} \log(1-p) \right) = 1/3 - \sum_{k=1}^\infty \frac{4 p^k}{(k+1)(k+2)(k+3)} \leq \frac{1}{3}.
\]
Therefore, 
\[
 \Var_\w\left(\fl{c_\w \tau} (E_\w F^{(j)}) - \b \tau \right) \leq ( E_\w S)^2 + \frac{(E_\w F^{(1)})^2}{3}.
\]
Recalling \eqref{couple3var}, we obtain that
\begin{align*}
 \Var_\w(T_\nu - \b \tau) &\leq \Var_\w( S) + ( E_\w S)^2 + \frac{(E_\w F^{(1)})^2}{3} + \Var_\w( F^{(1)} ) \left( E_\w \fl{c_\w \tau} \right) 
\end{align*}
Since \eqref{couplingsum} implies that
\[
 \Var_\w(T_\nu) = \Var_\w (S) + \Var_\w\left( \sum_{i=1}^{N} F^{(i)} \right) = \Var_\w (S) + (E_\w F^{(1)})^2 \Var_\w (N) + \Var_\w (F^{(1)}) (E_\w N),
\]
the bound  \eqref{varcoupleineq} follows. 
\end{proof}

The utility of the upper bound in Lemma \ref{varcouplelem} is that $E_\w F^{(1)}$ and $E_\w S$ are relatively small when $M_1$ is large. 
\begin{lem}\label{SFtails}
 For $0<\e<1$, 
\be\label{Stail}
 Q\left( E_\w S > n^{6\e/\k}, \, M_1 > n^{(1-\e)/\k} \right) = o(n^{-1}),
\ee
and
\be\label{Ftail}
 Q\left( E_\w F^{(1)} > n^{6\e/\k}, \, M_1 > n^{(1-\e)/\k} \right) = o(n^{-1}).
\ee
\end{lem}
The bound \eqref{Stail} on the tail decay of $E_\w S$ was proved in \cite[Corollary 4.2]{pzSL1}. The proof of \eqref{Ftail} is similar and involves straightforward but rather tedious computations using explicit formulas for quenched expectations and variances of hitting times conditioned on exiting an interval on a certain side. We defer the proof to Appendix \ref{details}. 

We now proceed to finish the proof of Lemma \ref{Tnun2bt} by extending
the coupling of $T_\nu$ with $\tau$ to all crossing times and showing
that the resulting coupling satisfies \eqref{e:req.couple}. 
As was done for $T_\nu$ in \eqref{couplingsum} we may decompose
$T_{\nu_i} - T_{\nu_{i-1}}$ so that, with the obvious notation, 
\[
 T_{\nu_i} - T_{\nu_{i-1}} = S_i + \sum_{j=1}^{N_i} F_i^{(j)}. 
\]
Lemma \ref{varcouplelem} tells us that 
\begin{align*}
 \sum_{i=1}^n&\Var_\w\left(T_{\nu_i} - T_{\nu_{i-1}} - \b_i \tau_i\right)\ind{M_i> n^{(1-\e)/\k}} \\
& \leq \sum_{i=1}^n \left( (E_\w S_i)^2 + \frac{(E_\w F_i^{(1)})^2}{3} + \Var_\w(T_{\nu_i}-T_{\nu_{i-1}}) - (E_\w F_i^{(1)})^2 \Var_\w (N_i) \right)\ind{M_i> n^{(1-\e)/\k}}. 
\end{align*}
An immediate consequence of Lemma \ref{SFtails} is that for any
$0<\e<1$, on an event of probability converging to one, all the $E_\w
S_i$ and $E_\w F_i^{(1)}$ with $i\leq n$ are less than $n^{6 \e/\k}$
when $M_1 > n^{(1-\e)/\k}$. Thus, by choosing $0<12\e/\k < 2/\k - 1$
we obtain that  
\[
 \lim_{n\ra\infty} Q\left( \frac{1}{n^{2/\k}} \sum_{i=1}^n \left( (E_\w S_i)^2 + \frac{(E_\w F_i^{(1)})^2}{3} \right)\ind{M_i> n^{(1-\e)/\k}} > \d \right) = 0, \quad \forall \d>0. 
\]
Therefore, to prove \eqref{e:req.couple} it is enough to show
\be\label{VarTnuEFVN}
 \lim_{n\ra\infty} Q\left( \frac{1}{n^{2/\k}} \sum_{i=1}^n \left( \Var_\w(T_{\nu_i}-T_{\nu_{i-1}}) - (E_\w F_i^{(1)})^2 \Var_\w N_i  \right)\ind{M_i> n^{(1-\e)/\k}} > \d \right) = 0, \quad \forall \d>0. 
\ee
In \cite{pzSL1}, it was shown that, when $M_1$ is large, $\b_1^2 =
(E_\w T_\nu)^2$ is comparable to $\Var_\w T_{\nu_1}$. In fact, as was
shown in the proof of Corollary 5.6 in \cite{pzSL1},  
\[
 \lim_{n\ra\infty} Q\left( n^{-2/\k} \left| \sum_{i=1}^n \left( \Var_\w (T_{\nu_i} - T_{\nu_{i-1}}) - \b_i^2 \right) \ind{M_i > n^{(1-\e)/\k}} \right| > \d \right) = 0, \quad \forall\d>0.  
\]
Therefore it only remains to show that
\[
 \lim_{n\ra\infty} Q\left( n^{-2/\k}  \sum_{i=1}^n \left( \b_i^2 - (E_\w F_i^{(1)})^2 \Var_\w (N_i) \right)\ind{M_i> n^{(1-\e)/\k}} > \d \right) = 0, \quad \forall\d>0. 
\]
Note that by \eqref{beq} 
\begin{align*}
 \b^2 - (E_\w F^{(1)})^2 \Var_\w (N) &= (E_\w S)^2 + 2 (E_\w S)(E_\w
 F^{(1)})(E_\w N)  - (E_\w F^{(1)})^2(E_\w N^2) \\ 
&\leq (E_\w S)^2 + 2 (E_\w S)(E_\w T_\nu).
\end{align*}
On the event where $E_\w S_i \leq n^{6\e/\k}$ for all $i\leq n$ with
$M_i> n^{(1-\e)/\k}$ we have 
\begin{align*}
\sum_{i=1}^n \left( \b_i^2 - (E_\w F_i^{(1)})^2 \Var_\w (N_i) \right)
\ind{M_i> n^{(1-\e)/\k}}  
& \leq n^{1+12\e/\k}+2n^{6\e/\k}\sum_{i=1}^n E_\w^{\nu_{i-1}} T_{\nu_i} \\
&= n^{1+12\e/\k}+2n^{6\e/\k}E_\w T_{\nu_n}. 
\end{align*}
Again, applying Lemma \ref{SFtails} with $0<12\e/\k < 2/\k -1$, we
see that for any $\d>0$, 
\begin{align*}
& \limsup_{n\ra\infty} Q\left( n^{-2/\k} \sum_{i=1}^n \left( \b_i^2 -
  (E_\w F_i^{(1)})^2 \Var_\w N_i \right) \ind{M_i> n^{(1-\e)/\k}} > \d \right) \\
&\quad\leq \limsup_{n\ra\infty} Q\left( n^{-2/\k+6\e/\k} E_\w
  T_{\nu_n} > \frac{\d}{2} \right),  
\end{align*}
and so the proof will be complete once we show 
that $n^{-2/\k + \e} E_\w T_{\nu_n} = n^{-2/\k+\e} \sum_{i=1}^n \b_i$
converges in probability to 0 for $\e > 0$ small enough.  

If $\k<1$, then since $n^{-1/\k} E_\w T_{\nu_n}$ converges in distribution \cite[Theorem 1.1]{pzSL1}, choosing $\e<1/\k$ works. If $\k>1$ then since 
$E_\w T_{\nu_n} = \sum_{i=1}^n \b_i$ and the $\b_i$ are stationary and
integrable under $Q$ (see \eqref{ETnutail}), the ergodic theorem
implies that $n^{-1}E_\w T_{\nu_n}$ converges and, hence, choosing   
$\e < 2/\k -1$ works. Finally, when $\k=1$ it follows from
\eqref{ETnutail} that for any $0<p<1$, $E_Q(\sum_{i=1}^n \b_i)^p\leq
a_p n$ for some $a_p\in (0,\infty)$, so choosing $\e<1$ works. 
\end{proof}

We conclude this section by noting that with a few minor modifications
of the proof of Proposition \ref{bigtransfer} we can obtain the
following analog in the case of the averaged centering.  
\begin{prop}\label{bigtransferA}
  For $\w \in \Omega$, suppose that $P_\w$ is expanded so that there
 exists a sequence $\tau_i$ which, under $P_\w$, is an 
 i.i.d.\ sequence of mean 1 exponential random variables. Let
 $\s_{n,\w} \in \mathcal{M}_1$ be defined by 
\be\label{sdefA}
 \s_{n,\w}(\cdot) = 
\begin{cases}
 P_\w\left( \frac{1}{n^{1/\k}} \sum_{i=1}^n \b_i \tau_i \in \cdot \right) & \k < 1 \\
P_\w\left( \frac{1}{n} \sum_{i=1}^n (\b_i \tau_i - D'(n)) \in \cdot \right) & \k = 1 \\
P_\w\left( \frac{1}{n^{1/\k}} \sum_{i=1}^n (\b_i \tau_i - \bar\b) \in \cdot \right) & \k \in (1,2), 
\end{cases}
\ee
where $D'(n) = E_Q[ \b_1 \ind{\b_1 \leq \bar\nu n} ] \sim \tc \log (n)$ and $\bar\b = E_Q [\b_1] = E_Q [E_\w T_\nu]$. 
Let $c_{\l,\k}(\e)$ be as in Theorem \ref{wqlTnA}, and set
$\tilde{c}_{\l,1}(\e) = \int_\e^{\bar\nu} \l x^{-1} \, dx =
c_{\l,1}(\e) + \l \log(\bar\nu)$. If 
\[
 \s_{n,\w} \overset{Q}{\Lra}
\begin{cases}
 H(N_{\l,\k}) & \kappa < 1 \\
 \lim_{\e\ra 0^+} H_\e(N_{\l,1}) * \d_{-\tilde{c}_{\l,1}(\e)} & \kappa = 1 \\
 \lim_{\e\ra 0^+} H_\e(N_{\l,\k}) * \d_{-c_{\l,\k}(\e)} & \kappa \in(1,2)
\end{cases}
, 
\]
then
\[
\mu_{n,\w} \overset{P}{\Lra}
\begin{cases}
 H(N_{\l/{\bar\nu},\k}) & \kappa < 1 \\
 \lim_{\e\ra 0^+} H_\e(N_{\l/{\bar{\nu}},\k}) * \d_{-c_{\l,\k}(\e)} & \kappa \in[1,2)
\end{cases}
,
\]
where ${\mu}_{n,\w}$ is as in Theorem \ref{wqlTnA}.
\end{prop}
\begin{rem}\label{DDprime}
 In the case $\kappa=1$, the relation between the sequences $D(n)$ and $D'(n)$ can be given by 
\[
 D(n) = \frac{\fl{n/\bar{\nu}}}{n} D'( \fl{n/\bar{\nu}}) = \frac{\fl{n/\bar{\nu}}}{n} E_Q\left[ \b_1 \ind{\b_1 \leq \bar{\nu}\fl{n/\bar{\nu}}} \right].  
\]
\end{rem}

\section{Analysis of the crossing times}\label{betaanalysis}

By Propositions \ref{bigtransfer} and \ref{bigtransferA}, our work is 
reduced to studying the distribution of a random mixture of
exponential random variables, where the random coefficients are the
average crossing times $\b_i = E_\w^{\nu_{i-1}} T_{\nu_i}$ in
\eqref{bdef}. The following proposition, which is the main result of
this section, establishes a Poisson limit of point processes arising from the random coefficients $\b_i$. 
\begin{prop}\label{PPconvergence}
For $n\geq 1$ let $N_{n,\w}$ be a point process defined by
\be\label{Nndef}
 N_{n,\w} = \sum_{i=1}^n \d_{\b_i/n^{1/\k}}. 
\ee
Then, under the measure $Q$, 
$N_{n,\w}$ converges weakly in the space $\mathcal{M}_p$ to a
non-homogeneous Poisson point process with intensity $\l
x^{-\k-1}$, where $\l=\tc \k$ and $\tc$ is the constant in \eqref{ETnutail}. 
That is, $N_{n,\w} \overset{Q}{\Lra} N_{\l,\k}$.  
\end{prop}
\begin{proof}
For a point process $\zeta = \sum_{i\geq 1} \d_{x_i}\in \mathcal{M}_p$
and a function $f:\, (0,\infty] \ra \R_+$, define the Laplace
 functional $\zeta(f) = \sum_{i\geq 1} f(x_i)$.  Since the weak
 convergence in the space $\mathcal{M}_p$ is equivalent to convergence
 of the Laplace functionals evaluated at all continuous functions with
 compact support of the type $[\delta,\infty]$ for some $\delta>0$
 (see Proposition 3.19 in \cite{rEVRVPP}), the 
 statement of the proposition will follow once we check that for any such
 $f$ 
\be \label{e:Laplace.conv}
 \lim_{n\ra\infty} E_Q\left[ e^{-N_{n,\w}(f) } \right] = \exp
 \left\{ - \int_0^\infty  (1-e^{-f(x)}) \l  x^{-\k-1} \, dx \right\}.   
\ee
\begin{rem} \label{rk:lipshitz}
An inspection of the argument of Propositions 3.16 and 3.19 in \cite{rEVRVPP} reveals
that the convergence in \eqref{e:Laplace.conv} for all continuous
functions with  compact support as above will follow once it is
checked for such functions that are, in addition, Lipschitz
continuous on $(0,\infty)$. 
\end{rem}

Recall from \eqref{ETnutail} that $Q(\b_1 > x) \sim \tc x^{-\k}$. Thus, if the $(\b_i)$ were i.i.d., the conclusion of
the proposition would follow immediately; see e.g. Proposition 3.21 in
\cite{rEVRVPP}. Since the sequence $(\b_i)$ is only stationary under
$Q$, our strategy is to show that the dependence between the $(\b_i)$ is
weak enough so that the point process $N_{n,\w}$ converges weakly to
the same limit as if the $(\b_i)$ were i.i.d.  

Recalling the notation in \eqref{pirwdef} and \eqref{rwdef} and the formula for quenched expectations of hitting times in \eqref{QETform}, we may write 
\begin{align*}
 \b_i = E_\w^{\nu_{i-1}}T_{\nu_i} &= \nu_i - \nu_{i-1} + 2 \sum_{j=\nu_{i-1}}^{\nu_i - 1} W_j \\
&= \nu_i - \nu_{i-1} + 2\sum_{j=\nu_{i-1}}^{\nu_i - 1} W_{\nu_{i-1},j} + 2W_{\nu_{i-1}-1} R_{\nu_{i-1},\nu_{i}-1}. 
\end{align*}
Thus, $\b_i = A_i Z_i + Y_i$, where
\[
 A_i = W_{\nu_{i-1}-1}, \quad Z_i = 2 R_{\nu_{i-1},\nu_i-1}, \quad \text{and}\quad Y_i = \nu_i - \nu_{i-1} + 2\sum_{j=\nu_{i-1}}^{\nu_i - 1} W_{\nu_{i-1},j} .
\]
Note that $Y_i$ and $Z_i$ only depend on the environment from
$\nu_{i-1}$ to $\nu_i - 1$, and therefore $\{(Y_i,Z_i)\}_{i\geq 1}$ is
an i.i.d.\ sequence of random variables with the same distribution as 
$$(Y_1,Z_1) = (\nu + 2 \sum_{j=0}^{\nu-1} W_{0,j}, \, 2 R_{0,\nu-1}).$$ Also,
note that the sequence $\{A_i\}_{i\geq 1}$ is stationary under the
measure $Q$.  From this decomposition of $\b_i$ we can see that 
the reason $(\b_i)$ is not an i.i.d.\ sequence is that the sequence
$(A_i)$ is not i.i.d.  The random variables $(A_i)$ all have the same
distribution under $Q$ as $A_1 = W_{-1}$. Furthermore, $W_{-1}$ has
exponential tails under $Q$. That is, there exist constants $C,C'>0$
such that 
\be\label{Wtail}
 Q(W_{-1} > x) \leq C' e^{-C x}; 
\ee
see Lemma 4.2.2 in \cite{pThesis}. In addition, $W_{-1}$ can be very
well approximated  by $W_{-j,-1}$ for large  $j$. That is, there exist
constants  $C_1, C_2, C_3>0$    such that  for every $j=1,2,\ldots$, 
\be\label{Wcutoff}
 Q( W_{-1} -  W_{-j,-1} > e^{-C_1 j} ) \leq C_2 e^{-C_3 j}. 
\ee
To see this, defining the ladder locations $\nu_{-k}$ to the left of the origin in the natural way (see \cite{pzSL1}), observe that for any $c>0$, 
\begin{align*}
 Q( W_{-1} - W_{\nu_{-k},-1} > e^{-c k} ) &\leq e^{ck} E_Q[ W_{-1} - W_{\nu_{-k},-1} ] \\
&= e^{ck} E_Q[ \Pi_{\nu_{-k},-1} W_{\nu_{-k}-1} ] = e^{ck} E_Q[ \Pi_{0,\nu-1} ]^k E_Q[ W_1]. 
\end{align*}
Since $E_Q[\Pi_{0,\nu-1}] < 1$ by the definition of the ladder
locations, choosing $c$ small enough gives us an exponential bound $Q( W_{-1} -
W_{\nu_{-k},-1} > e^{-c k} ) \leq C' e^{-C k}$, $k=1,2,\ldots$ for
some positive $C,C'$. The bound \eqref{Wcutoff} now follows by
writing, for $a>0$, 
$$
Q( W_{-1} -  W_{-j,-1} > e^{-c j} ) 
\leq Q( W_{-1} -  W_{-\nu_{aj},-1} > e^{-c j} ) 
+ Q(\nu_{aj}>j),
$$
and noticing that, by \eqref{nutail}, for $a>0$ small enough, the
latter probability is exponentially small as a function of $j$.


Keeping the exponential bounds \eqref{Wtail} and \eqref{Wcutoff} in
mind, we modify the sequence of the crossing times in order to reduce
the dependence. For $n\geq 1$ we set $A_i^{(n)} =
W_{\nu_{i-1}-\fl{\sqrt{n}},\nu_{i-1}-1}$ and 
$\b_i^{(n)} = A_i^{(n)} Z_i + Y_i$, $i=1,2,\ldots$. Notice that 
$\b_i^{(n)}$ and $\b_j^{(n)}$ are independent if $|i-j| >
\sqrt{n}$. Next, we give a comparison of $\b_i^{(n)}$ with $\b_i$ that will allow us to analyze the 
tail behaviour of the random
variables $(\b_i^{(n)})$. 
\begin{lem}\label{bcutoff}
There exist constants, $C,C'>0$ such that
\[
 Q\left( \b_1 - \b_1^{(n)} > e^{-n^{1/4}} \right) \leq C e^{-C'
   \sqrt{n}}, \ n=1,2,\ldots.  
\]
\end{lem}
\begin{proof}
 From the decompositions of $\b_i$ and $\b_i^{(n)}$ we obtain that
 $\b_i - \b_i^{(n)} = (A_i - A_i^{(n)}) Z_i$.  
Note that $Z_1 = 2R_{0,\nu-1} \leq 2R_0$. By \eqref{ETnutail} there
exists a constant $C$ such that $Q(Z_1 > x) \leq C x^{-\k}$ for all
$x>0$. Therefore, for any $x>0$ 
\begin{align*}
 Q\left( \b_1 - \b_1^{(n)} > x \right) &\leq Q\left( A_1 - A_1^{(n)}
 > e^{-C_1 \sqrt{n}} \right) + Q\left( Z_1 >  e^{C_1 \sqrt{n}} x
 \right) \\ 
&\leq C_2 e^{-C_3 \sqrt{n}} + C e^{-C_1 \kappa \sqrt{n}} x^{-\k}. 
\end{align*}
Choosing $x = e^{-n^{1/4}}$ completes the proof. 
\end{proof}

Based on the truncated crossing times $(\b_i^{(n)})$ we define a
sequence of point processes by  
\[
 N_{n,\w}^{(n)} = \sum_{i\geq 1} \d_{\b_i^{(n)}/n^{1/\k}},
 \ n=1,2,\ldots. 
\]
\begin{lem}\label{cutoffPPconv}
 $N_{n,\w}^{(n)} \overset{Q}{\Lra}N_{\l,\k}$ as $n\ra\infty$ for
  $\l=\tc$, the constant in \eqref{ETnutail}. 
\end{lem}
\begin{proof}
Let $f:\, (0,\infty] \ra \R_+$ be a continuous functon vanishing for
all $0<x<\d$ for some $\d>0$, and Lipshitz on the interval
$(\d,\infty)$. We will prove the following analogue of
\eqref{e:Laplace.conv}:  
\be \label{e:Laplace.conv.1}
 \lim_{n\ra\infty} E_Q\left[ e^{-N_{n,\w}^{(n)}(f) } \right] = \exp
 \left\{ - \int_0^\infty  (1-e^{-f(x)}) \l  x^{-\k-1} \, dx \right\}.   
\ee
According to Remark \ref{rk:lipshitz}, this will give us the claim of
the lemma. 

For $0<\tau<1$ we define a sequence of random random variables 
$$
K_n(\tau) = \text{card}\bigl\{ i=1,\ldots, n:\ \text{both} \ \b_i^{(n)}>\d
n^{1/\k} \ \text{and} \ \b_j^{(n)}>\d n^{1/\k} 
$$
$$
\text{for some} \
i+1\leq j\leq i+\tau n, \, j\leq n.  \bigr\}.
$$
We claim that
\be \label{e:Kn}
\lim_{\tau\to 0} \limsup_{n\to\infty} Q(K_n(\tau)>0)=0.
\ee
To see this, let $0<\e<1$, and consider a sequence of events 
$$
B_n(\e) = \bigl\{ \text{for some $i=1,\ldots,n$}, \ \b_i^{(n)}>\d
n^{1/\k}\ \text{but} \ \max(Y_i,Z_i)\leq \e n^{1/\k}\bigr\}.
$$
Since by \eqref{ETnutail} there
exists a constant $C$ such that $Q(\max(Y_1,Z_1) > x) \leq C x^{-\k}$
for all $x>0$, while by \eqref{Wtail} the random variable $A_1$ has
an exponentially fast decaying tail, we see that
\begin{align*}
Q(B_n(\e)) & \leq nQ\bigl( \max(Y_1,Z_1)\leq \e n^{1/\k}, \, \b_1^{(n)}>\d
n^{1/\k}\bigr) \\
& \leq nQ\bigl( \max(Y_1,Z_1)\leq \e n^{1/\k}, \, (A_1+1)\max(Y_1,Z_1) 
>\d n^{1/\k}\bigr)\\
& = O\Bigl( nQ(\max(Y_1,Z_1) > \d n^{1/\k}) E_Q\bigl( (A_1+1)^\k
\one(A_1+1 >\d/\e)\bigr)\Bigr) \\
& = O\Bigl( \d^{-\k} E_Q\bigl( (A_1+1)^\k
\one(A_1+1 >\d/\e)\bigr)\Bigr) 
\end{align*}
as in, for example, Breiman's lemma (\cite{breiman:1965}). Therefore,
\be \label{e:Bn}
\lim_{\e\to 0} \limsup_{n\to\infty} Q(B_n(\e))=0.
\ee
For $\tau,\e>0$
\begin{align*}
Q(K_n(\tau)>0) & \leq Q(B_n(\e))
+ Q\bigl( \text{for some $i=1,\ldots,n$, some $i+1\leq j\leq i+\tau
n$,} \\
& \qquad\qquad\qquad\qquad \max(Y_i,Z_i)> \e n^{1/\k} \ \text{and} \ \max(Y_j,Z_j)> 
\e n^{1/\k}\bigr) \\
& \leq Q(B_n(\e)) + \tau n^2 \bigl( Q(\max(Y_1,Z_1)> \e
n^{1/\k}\bigr)^2 \\
&\leq Q(B_n(\e)) + C^2\e^{-2\k}\tau.
\end{align*}
We conclude that 
$$
\lim_{\tau\to 0} \limsup_{n\to\infty} Q(K_n(\tau)>0) \leq 
\limsup_{n\to\infty} Q(B_n(\e)),  
$$
and so \eqref{e:Kn} follows from \eqref{e:Bn}. 

Fix, for a moment, $\e>0$ and take $\tau>0$ such that for some $n_0$
we have $Q(K_n(\tau)>0)\leq \e$ for all $n\geq n_0$; this is possible
by \eqref{e:Kn}. Consider the random sets 
$$
D_n =  \{ i=1,\ldots,n:\, \b_i^{(n)}>\d n^{1/\k}\}.
$$
Since $f(x)=0$ if $x\leq \d$, we can write
\begin{align}
E_Q\left[ e^{-N_{n,\w}^{(n)}(f) } \right] & = 
E_Q\exp\left\{ -\sum_{i\in D_n}
f\bigl(\b_i^{(n)}/n^{1/\k}\bigr)\right\}  \label{e:N.decomp}  \\
&= E_Q\left[ \exp\left\{ -\sum_{i\in D_n}
f\bigl(\b_i^{(n)}/n^{1/\k}\bigr)\right\} \one( K_n(\tau)=0)\right]
\nonumber \\
&\qquad + E_Q\left[ \exp\left\{ -\sum_{i\in D_n}
f\bigl(\b_i^{(n)}/n^{1/\k}\bigr)\right\} \one( K_n(\tau)>0)\right]
\nonumber \\
&:= H_n^{(1)} + H_n^{(2)}. \nonumber 
\end{align}
By the choice of $\tau$,
\be \label{e:H2}
\limsup_{n\to\infty} H_n^{(2)}\leq \limsup_{n\to\infty} Q(K_n(\tau)>0)\leq \e.
\ee
Moreover, given the event $\{  K_n(\tau)=0\}$, the points in the
random set $D_n$ are separated, for large $n$, by more than $\sqrt{n}$
and, hence, given also the random set $D_n$, the random variables
$\b_i^{(n)}, \, i\in D_n$ are independent, each one with the
corresponding conditional distribution. That is, 
$$
H_n^{(1)} = Q(K_n(\tau)=0)
E_Q\left\{ \Bigl[ E_Q \Bigl( \exp\bigl\{
  -f\bigl(\b_1^{(n)}/n^{1/\k}\bigr)\bigr\} \big|\b_1^{(n)}>\d
  n^{1/\k}\Bigr)\Bigr]^{{\rm card} D_n}\bigg| K_n(\tau)=0\right\}.
$$
The power law \eqref{ETnutail} and Lemma \ref{bcutoff} show the weak
convergence to the Pareto distribution
$$
Q\Bigl( \b_1^{(n)}/n^{1/\k}>t\big| \b_1^{(n)}>\d   n^{1/\k}\Bigr) 
\to (t/\d)^{-\k}
$$
for $t\geq \d$, and so by the bounded convergence theorem,
$$
 E_Q \Bigl( \exp\bigl\{
  -f\bigl(\b_1^{(n)}/n^{1/\k}\bigr)\bigr\} \big|\b_1^{(n)}>\d
  n^{1/\k}\Bigr) \to \int_1^\infty e^{-f(\d t)}\k t^{-(\k+1)}\, dt.
$$
Now the claim \eqref{e:Laplace.conv.1} follows from
\eqref{e:N.decomp}, \eqref{e:H2} and the following limiting statement:
for the constant $\tc$ in \eqref{ETnutail}, 
\begin{align} 
& \exp\bigl\{ -\tc (1-\alpha)\delta^{-\k}\bigr\} \leq \lim_{\tau\to
    0}\liminf_{n\to\infty} E_Q \left( \alpha^{{\rm card} D_n}\Big| 
K_n(\tau)=0\right) \label{e:mgf.card} \\
 = & \lim_{\tau\to 0}\limsup_{n\to\infty} E_Q \left(
\alpha^{{\rm card} D_n}\Big| K_n(\tau)=0\right)
\leq  \exp\bigl\{ -\tc (1-\alpha)\delta^{-\k}\bigr\} \nonumber
\end{align}
for all $0<\alpha<1$. In order to complete the proof of the lemma it,
therefore, remains to prove \eqref{e:mgf.card}.

We split the set $\{1,2\ldots, n\}$ into a union of the following
sets. Let
\begin{align*}
& I_{1.n} = \{1, \ldots, [n^{3/4}]\}, \ J_{1.n} = \{ [n^{3/4}]+1,
\ldots, [n^{3/4}] + [n^{2/3}], \\
& I_{2.n} = \{[n^{3/4}] + [n^{2/3}]+1, \ldots, 2[n^{3/4}]+
      [n^{2/3}]\}, \\
& J_{2.n} = \{ 2[n^{3/4}]+ [n^{2/3}]+1, \ldots, 2[n^{3/4}] +
2[n^{2/3}]\},
\end{align*}
etc. (the last interval can be a bit shorter than the rest). Clearly,
the cardinality $m_n$ of the union of all intervals  $J_{k.n}$
satisfies $m_n/n \to 0$ as $n\to\infty$. We write
$D_n = D_n^{(I)}\cup D_n^{(J)}$, where $D_n^{(I)}$ (resp. $D_n^{(J)}$)
contains all the points of $D_n$ that are in one of the
intervals $I_{k.n}$ (resp. $J_{k.n}$). Observe that the intervals
$I_{k.n}$ are separated  by more that $\sqrt{n}$, so for $i$ and $j$
in two different of this type, $\b_i^{(n)}$ and $\b_j^{(n)}$ are
independent. We have
\begin{align*} 
E_Q \left( \alpha^{{\rm card} D_n} \one\bigl( K_n(\tau)=0\bigr)\right) 
&\leq E_Q \left( \alpha^{{\rm card} D_n^{(I)}} \right) \\ 
&=  \left( E_Q\alpha^{{\rm Card}(D_n\cap I_{1.n})}\right)^{[n/([n^{3/4}]+ [n^{2/3}])]}.
\end{align*}
Repeating the argument leading to \eqref{e:Kn} (that shows that
$\b_i^{(n)}$ and $\b_j^{(n)}$ can both exceed $\d n^{1/\k}$ for
$0<|i-j|\leq n^{3/4}$ only on an event of a vanishing probability)
tells us that 
$$
Q\bigl( {\rm Card} (D_n\cap   I_{1.n})=1\bigr) 
\sim n^{3/4} Q\bigl( \b_1^{(n)}>\d n^{1/\k}\bigr)
\sim n^{3/4} \tc \d^{-\k}n^{-1} = \tc \d^{-\k}n^{-1/4},
$$
$$
Q\bigl( {\rm Card} (D_n\cap   I_{1.n})>1\bigr) =o(n^{-1/4}),
$$
Therefore,
$$
E_Q\alpha^{{\rm Card}(D_n\cap I_{1.n})} = 1-(1-\alpha)
\tc \d^{-\k}n^{-1/4} + o(n^{-1/4}),
$$
implying that  
$$
\limsup_{n\to\infty} E_Q \left(
\alpha^{{\rm card} D_n}\Big| K_n(\tau)=0\right)
\leq
\frac{1}{Q(K_n(\tau)=0)}\exp\bigl\{ -\tc (1-\alpha)\delta^{-\k}\bigr\},
$$
and the upper limit part in \eqref{e:mgf.card} follows from
\eqref{e:Kn}. 

Similarly, 
\begin{align*} 
& E_Q \left( \alpha^{{\rm card} D_n} \one\bigl(
  K_n(\tau)=0\bigr)\right)
\geq E_Q \left( \alpha^{{\rm card} D_n^{(I)}} \one\bigl( K_n(\tau)=0,
D_n^{(J)}=0\bigr)\right) \\
\geq & E_Q \left( \alpha^{{\rm card} D_n^{(I)}} \right) -
Q(K_n(\tau)>0) - Q(D_n^{(J)}>0).
\end{align*}
The last term vanishes in the limit since $m_n/n \to 0$. Therefore,
$$
\liminf_{n\to\infty} E_Q \left(
\alpha^{{\rm card} D_n}\Big| K_n(\tau)=0\right)
\geq
\exp\bigl\{ -\tc (1-\alpha)\delta^{-\k}\bigr\} - Q(K_n(\tau)>0),
$$
and the lower limit part in \eqref{e:mgf.card} follows from
\eqref{e:Kn} as well. 
\end{proof}

Now we are ready to finish the proof of Proposition
\ref{PPconvergence}, which we accomplish by checking
\eqref{e:Laplace.conv} for nonnegative continuous functions $f$ on
$(0,\infty]$ 
  with compact support that are Lipschitz continuous  on $(0,\infty)$. 
For any such function $f$, 
\begin{align*}
 E\left[e^{-N_{n,\w}(f)} \right] &= E\left[ \exp\left\{ - \sum_{i=1}^n f(\b_i/n^{1/\k}) \right\} \right] \\
&= E\left[e^{-N_{n,\w}^{(n)}(f)} \exp\left\{-\sum_{i=1}^n\left(
   f(\b_i/n^{1/\k}) - f(\b_i^{(n)}/n^{1/\k}) \right) \right\}
   \right]. 
\end{align*}
Now, let 
\[
 \Omega_n := \left\{ \w \in \Omega: \, \b_i - \b_i^{(n)} \leq e^{-n^{1/4}}, \, \forall i=1,2,\ldots n \right\}
\]
Lemma \ref{bcutoff} implies that $Q(\Omega_n^c) \ra 0$ as $n\ra\infty$. 
Since $f$ is Lipschitz with some constant $c$, on the event $\Omega_n$
we have 
\begin{align*}
\left| \sum_{i=1}^n\left( f(\b_i/n^{1/\k}) - f(\b_i^{(n)}/n^{1/\k}) \right) \right|
&\leq \frac{c}{n^{1/\k}} \sum_{i=1}^n |\b_i - \b_i^{(n)}| \\
&\leq c n^{1-1/\k} e^{-n^{1/4}}, 
\end{align*}
and so by Lemma \ref{cutoffPPconv}
\[
 \lim_{n\ra\infty} E\left[e^{-N_{n,\w}(f)} \right] = \lim_{n\ra\infty} E\left[e^{-N_{n,\w}^{(n)}(f)} \mathbf{1}_{\Omega_n} \right] = \Ev \left[ e^{-N_{\l,\k}(f) } \right],
\]
proving \eqref{e:Laplace.conv}. 
\end{proof}

In addition to the already established convergence of the point
processes $(N_{n,\w})$, in the sequel we will also need the following
tail bound on the sums of the average crossing times $\b_i$ that are not extremely large. 
\begin{lem}\label{ble}
 Let $\k\in[1,2)$. Then for any $\d>0$, 
\[
 \lim_{\e\ra 0^+} \limsup_{n\ra\infty} Q\left( \frac{1}{n^{1/\k}} \left| \sum_{i=1}^n \left( \b_i \ind{\b_i \leq \e n^{1/\k}} - E_Q[ \b_1 \ind{\b_1 \leq \e n^{1/\k}}] \right) \right| \geq \d \right) = 0. 
\]
\end{lem}
\begin{proof} Clearly, $\b_i\ind{\b_i \leq \e n^{1/\k}} = \b_i \wedge
  \e n^{1/\k} - \e n^{1/\k} \ind{\b_i > \e n^{1/\k}}$. Therefore,  
\begin{align}
& Q\left( \frac{1}{n^{1/\k}} \left| \sum_{i=1}^n \left( \b_i \ind{\b_i \leq \e n^{1/\k}} - E_Q[ \b_1 \ind{\b_1 \leq \e n^{1/\k}}] \right) \right| \geq \d \right) \nonumber \\
&\quad\leq Q\left( \frac{1}{n^{1/\k}} \left| \sum_{i=1}^n \left( \b_i \wedge \e n^{1/\k} - E_Q[ \b_1 \wedge \e n^{1/\k} ] \right) \right| \geq \d/2 \right) \label{truncterm} \\
&\quad\qquad + Q\left( \e \left| \sum_{i=1}^n \ind{\b_i > \e n^{1/\k}}
  - n Q( \b_1 > \e n^{1/\k})  \right| \geq \d/2 \right).  \label{poissonterm}
\end{align}
We will first handle the term in \eqref{poissonterm}. 
For $\e>0$, let $G_\e : \mathcal{M}_p \ra \Z_+$ be defined by
$G_\e(\zeta) = \sum_{i\geq 1} \ind{x_i > \e}$ when $\zeta =
\sum_{i\geq 1}\d_{x_i}$. Then, since $G_\e$ is continuous on the set
$\mathcal{M}_p^{(\e)} = \{\zeta(\{\e\} = 0 \}$, we conclude by
Proposition \ref{PPconvergence} and the continuous mapping theorem that  
$\sum_{i=1}^n \ind{\b_i > \e n^{1/\k}} = G_\e(N_{n,\w}) \Lra
G_\e(N_{\l,\k})$. Further, it follows from \eqref{ETnutail} that $nQ(
\b_1 > \e n^{1/\k}) \ra \tc \e^{-\k} =  \Ev[ G_\e(N_{\l,\k})]$ as
$n\ra\infty$. Now, since $G_\e(N_{\l,\k})$ has Poisson distribution
with mean $\l\e^{-\k}/\k$, we see that 
\begin{align*}
& \lim_{\e\ra 0} \limsup_{n\ra\infty} Q\left( \e \left| \sum_{i=1}^n \ind{\b_i > \e n^{1/\k}} - n Q( \b_1 > \e n^{1/\k})  \right| \geq \d/2 \right) \\
&\quad \leq \lim_{\e\ra 0} \Pv\left( |G_\e(N_{\l,\k}) - \Ev[ G_\e(N_{\l,\k})] | \geq \frac{\d}{2\e} \right) \\
&\quad \leq \lim_{\e\ra 0} \frac{4\e^2}{\d^2} \Varv( G_\e(N_{\l,\k}) ) = \lim_{\e\ra 0} \frac{4\e^{2-\k}\l}{\d^2\k} = 0. 
\end{align*}

Next, we estimate the probability in \eqref{truncterm}. 
By Chebychev's inequality and the fact that the $\b_i$ are stationary
under $Q$,  this probability is bounded above by
\begin{align}
& \frac{4}{\d^2 n^{2/\k}} \Var_Q\left( \sum_{i=1}^n \b_i \wedge \e n^{1/\k} \right) \nonumber \\
&\quad= \frac{4}{\d^2 n^{2/\k}}n \Var_Q(\b_1 \wedge \e n^{1/\k}) + \frac{8}{\d^2 n^{2/\k}} \sum_{k=1}^n(n-k) \Cov_Q(\b_1 \wedge \e n^{1/\k}, \b_{k+1}\wedge \e n^{1/\k} ). \label{VarCov}
\end{align}
Now, the tail decay \eqref{ETnutail} of $\b_1$ and Karamata's theorem
(see p. 17 in \cite{rEVRVPP}) imply that 
\[
\limsup_{n\ra\infty} n^{-(2/\k-1)}\Var_Q(\b_1 \wedge \e n^{1/\k}) \leq
\lim_{n\ra\infty}  n^{-(2/\k-1)} E_Q[ \b_1^2 \wedge \e^2 n^{2/\k} ] 
= \frac{2 \tc}{2-\k} \e^{2-\k}.
\]
Since $\k<2$ this vanishes as $\e\ra 0$ and so 
to finish the proof of the lemma it is enough to show that 
\be\label{Covterms}
 \lim_{\e\ra 0} \limsup_{n\ra\infty} \frac{1}{n^{2/\k}} \sum_{k=1}^n(n-k) \Cov_Q(\b_1 \wedge \e n^{1/\k}, \b_{k+1}\wedge \e n^{1/\k} ) = 0.
\ee

To bound the covariance terms, we use \eqref{QETform} to write 
\begin{align*}
 \b_{k+1} &= \sum_{j=\nu_{k}}^{\nu_{k+1}-1} (1+2W_j) \\
&= \nu_{k+1} - \nu_k + 2 \sum_{j=\nu_{k}}^{\nu_{k+1}-1} W_{\nu_1,j} + 2 W_{\nu_1-1} \Pi_{\nu_1,\nu_k-1} R_{\nu_k, \nu_{k+1}-1} \\
&=: \tilde\b_{k+1} + 2 W_{\nu_1-1} \Pi_{\nu_1,\nu_k-1} R_{\nu_k, \nu_{k+1}-1}. 
\end{align*}
Note that $\tilde\b_{k+1}$ is independent of $\b_1$, so that for some
constant $C'$
\begin{align}
 \Cov_Q(\b_1 \wedge \e n^{1/\k}, \b_{k+1}\wedge \e n^{1/\k} ) &= \Cov_Q(\b_1 \wedge \e n^{1/\k}, \b_{k+1}\wedge \e n^{1/\k} -  \tilde\b_{k+1}\wedge \e n^{1/\k}) \nonumber \\
&\leq \sqrt{ \Var_Q( \b_1 \wedge \e n^{1/\k} ) } \sqrt{ \Var_Q( \b_{k+1}\wedge \e n^{1/\k} -  \tilde\b_{k+1}\wedge \e n^{1/\k}) } \nonumber \\
&\leq C' \e^{1-\k/2} n^{1/\k-1/2}\sqrt{ E_Q[ (\b_{k+1} -
     \tilde\b_{k+1})^2 \ind{ \tilde\b_{k+1} \leq \e n^{1/\k} } ] } \label{Covub}
\end{align}
for $n$ large enough. An examination of the formula for
$\tilde\b_{k+1}$ shows that $R_{\nu_k,\nu_{k+1}-1} \leq
\tilde\b_{k+1}$. Therefore,  
\begin{align}
& E_Q[ (\b_{k+1} -  \tilde\b_{k+1})^2 \ind{ \tilde\b_{k+1} \leq \e n^{1/\k} } ] 
= 4 E_Q\left[ W_{\nu_1-1}^2 \Pi_{\nu_1,\nu_k-1}^2 R_{\nu_k, \nu_{k+1}-1}^2 \ind{ \tilde\b_{k+1} \leq \e n^{1/\k} } \right] \nonumber \\
&\quad \leq 4 E_Q\left[ W_{\nu_1-1}^2 \right] E_Q\left[\Pi_{\nu_1,\nu_k-1}^2 \right] E_Q\left[ R_{\nu_k, \nu_{k+1}-1}^2 \ind{ R_{\nu_k, \nu_{k+1}-1} \leq \e n^{1/\k} } \right] \nonumber \\
&\quad = 4 E_Q\left[ W_{-1}^2 \right] E_Q\left[\Pi_{0,\nu-1}^2 \right]^{k-1} E_Q\left[ R_{0,\nu-1}^2 \ind{ R_{0,\nu-1} \leq \e n^{1/\k} } \right], \label{bbtilde}
\end{align}
where in the last step we used the invariance of the distribution $Q$
under shifts by the ladder locations $\nu_i$. Further, $E_Q[W_{-1}^2]<
\infty$ by \eqref{Wtail}, and $E_Q[\Pi_{0,\nu-1}] < 1$ by the
definition of the ladder locations. Also, since $R_{0,\nu-1} \leq
\b_1$, $E_Q\left[ R_{0,\nu-1}^2 \ind{ R_{0,\nu-1} \leq \e n^{1/\k} } \right]
\leq C' \e^{2-\k} n^{2/\k-1}$ for large $n$. Combining this with \eqref{Covub} and
\eqref{bbtilde} we see that for some $0<\rho<1$, 
\[
 \Cov_Q(\b_1 \wedge \e n^{1/\k}, \b_{k+1}\wedge \e n^{1/\k} ) \leq
 (C')^2 \e^{2-\k} n^{2/\k-1} \rho^k, 
\]
and this bound on the covariance is sufficient to prove
\eqref{Covterms}. This finishes the proof of the lemma. 
\end{proof}

We conclude this section by giving a corollary of Lemma \ref{ble} that is of independent interest. 
In \cite{pzSL1} it was shown that, if $0<\k<1$, then $n^{-1/\k} E_\w
T_{\nu_n} = n^{-1/\k} \sum_{i=1}^n \b_i$ converges in distribution to
a $\k$-stable random variable. 
The following corollary shows that $E_\w T_{\nu_n}$ has a stable limit law when $\k\in[1,2)$ as well. 
\begin{cor}\label{ETnunstable12}
 If $\k = 1$, then there exists a $b>0$ and a sequence $D''(n) = E[\b_1 \ind{\b_1 \leq n} ] \sim \tc
 \log n$ such that
\[
 \lim_{n\ra\infty} Q\left( \frac{E_\w T_{\nu_n} - n D''(n)}{n} \leq x \right) = L_{1,b}(x), \qquad \forall x\in\R. 
\]
If $\k\in(1,2)$, then 
\[
 \lim_{n\ra\infty} Q\left( \frac{E_\w T_{\nu_n} - nE_Q[E_\w T_{\nu_1}]}{n^{1/\k}} \leq x \right) = L_{\k,b}(x), \qquad \forall x\in\R.
\]
In both cases $b^\k = \l/\k$. 
\end{cor}
\begin{proof}
 This is a direct application of Proposition \ref{PPconvergence} and
Lemma \ref{ble} to Theorem 3.1 in \cite{dhStationaryPP}. 
\end{proof}

\section{Weak quenched limits of hitting times - quenched centering}\label{quenchedcentering}

Having done the necessary preperatory work in Sections \ref{Transfer} and \ref{betaanalysis} we are now ready to prove Theorem \ref{wqlTn}.
Recall, that by Proposition \ref{bigtransfer} it is enough to show
that $\bar{\s}_{n,\w} \overset{Q}{\Lra} \bar{H}(N_{\l,\k})$ for some
$\l>0$, where $\bar{\s}_{n,\w} = \bar{H}(N_{n,\w})$ is given in
\eqref{sdef}, while $\bar{H}$ and $N_{n,\w}$ are defined by
\eqref{barHdef} and \eqref{Nndef}, respectively. 
Since $N_{n,\w} \overset{Q}{\Lra} N_{\l,\k}$ by Proposition
\ref{PPconvergence}, if the mapping $\bar{H}:\mathcal{M}_p \ra
\mathcal{M}_1$ were continuous the statement of Theorem \ref{wqlTn} would
follow by the continuous mapping theorem. Unfortunately, $\bar{H}$ is
not a continuous mapping. To overcome this, we employ a truncation
technique. 

For $\e>0$ define the a mapping $\bar{H}_\e: \mathcal{M}_p \ra
\mathcal{M}_1$ by modifying the definition  \eqref{sdef} as follows: 
\be\label{barHe}
 \bar{H}_\e(\zeta)(\cdot) = \Pv\left( \sum_{i\geq 1} x_i(\tau_i-1)\ind{x_i > \e} \in \cdot \, \right), \quad\text{when } \zeta=\sum_{i\geq 1} \d_{x_i}. 
\ee
It turns out that this mapping is continuous on the relevant subset of
$\mathcal{M}_p$.  
\begin{lem}\label{barHecont}
$\bar{H}_\e$ is continuous on the set $\mathcal{M}_p^{(\e)} := \{ \zeta \in \mathcal{M}_p : \, \zeta(\{\e\}) = 0 \}$.
\end{lem}
\begin{proof}
%
Let $\zeta_n \overset{v}{\ra} \zeta \in \mathcal{M}_p^{(\e)}$. Then, by \cite[Proposition 3.13]{rEVRVPP} there exists an integer $M$ and a labelling of the points of $\zeta$ and $\zeta_n$ (for $n$ sufficiently large) such that 
\[
  \zeta( \cdot \cap (\e,\infty) ) = \sum_{i=1}^M \delta_{x_i}, 
\quad \text{and}\quad 
\zeta_n( \cdot \cap (\e,\infty) ) =
\sum_{i=1}^M \delta_{x_i^{(n)}}, 
\]
with $(x_1^{(n)}, x_2^{(n)}, \ldots x_M^{(n)}) \ra (x_1,x_2,\ldots x_M)$ as $n\ra\infty$.
Consequently,
\[
\lim_{n\ra\infty} \bar{H}_\e(\zeta_n)(\cdot)  = \lim_{n\ra\infty} \Pv \left( \sum_{i=1}^M x_i^{(n)} (
\tau_i - 1) \in \cdot \right) 
= \Pv \left( \sum_{i=1}^M x_i ( \tau_i - 1) \in \cdot \right) =
\bar{H}_\e(\zeta)(\cdot) 
\]
in the space $\mathcal{M}_1$.  
%
\end{proof}

\begin{proof}[Proof of Theorem \ref{wqlTn}]
Since $\Pv(N_{\l,\k} \notin \mathcal{M}_p^{(\e)} ) = 0$, Proposition
\ref{PPconvergence}, Lemma \ref{barHecont} 
 and the continuous mapping theorem \cite[Theorem 2.7]{bCOPM} imply
 that  for every $\e>0$, 
\be\label{ncon}
 \bar{H}_\e(N_{n,\w}) \overset{Q}{\Lra} \bar{H}_\e(N_{\l,\k}), \quad \text{as } n\ra\infty. 
\ee
Next, we claim that  
\be\label{econ}
 \lim_{\e\ra 0^+} \bar{H}_\e(N_{\l,\k}) = \bar{H}(N_{\l,\k}), \quad \Pv\text{-a.s.}
\ee
and 
\be\label{ipcon}
 \lim_{\e\ra 0} \limsup_{n\ra\infty} Q\left( \rho(\bar{H}_\e(N_{n,\w}), \bar{H}(N_{n,\w})) \geq \d \right) = 0, \quad \forall \d>0. 
\ee
By \cite[Theorem 3.2]{bCOPM} this will show that 
\[
\bar{\s}_{n,\w} = \bar{H}(N_{n,\w}) \overset{Q}{\Lra}
\bar{H}(N_{\l,\k}), 
\]
which, by Proposition \ref{bigtransfer}, is enough for the 
the conclusion of Theorem \ref{wqlTn}. Thus, it only remains to prove
\eqref{econ} and \eqref{ipcon}.  Since the claim \eqref{econ} follows
from the continuity of the map $ \bar{H}_2$ in the proof of Lemma
\ref{l:meas} in Appendix \ref{Hmeas}, we prove \eqref{ipcon}. 

Recall that for any two random variables $X$ and $Y$ defined on the
same probability space, with respective laws ${\mathcal L}_{X}$ and
${\mathcal L}_{Y}$, $\rho({\mathcal L}_{X}, {\mathcal L}_{Y})\leq
\bigl( E|X-Y|^2\bigr)^{1/3}$. Therefore, 
$$
\rho( \bar{H}_\e(N_{n,\w}), \bar{H}(N_{n,\w})) \leq \left( \frac{1}{n^{2/\k}}
\sum_{i=1}^n \beta_i^2 \ind{ \b_i/n^{1/\k} \leq \e } \right)^{1/3}
$$
and so by the Markov inequality, \eqref{ETnutail} and Karamata's theorem, 
\begin{align*}
\limsup_{n\to\infty} Q\left( \rho(\bar{H}_\e(N_{n,\w}), \bar{H}(N_{n,\w})) \geq \d \right) 
&\leq \limsup_{n\to\infty} Q\left( \frac{1}{n^{2/\k}} \sum_{i=1}^n \beta_i^2 \ind{ \b_i/n^{1/\k} \leq \e } \geq \d^3 \right) \\
&\leq \limsup_{n\to\infty} \frac{n^{1-2/\k}}{\d^3} E_Q[ \b_1^2
  \ind{\b_1 \leq \e n^{1/\k}} ] \\
& = \frac{\tc \k \e^{2-\k} }{(2-\k)\d^3}
\end{align*}
Since $\k <2$ the right hand side tends to $0$ as $\e\ra 0$. This completes the proof of \eqref{ipcon} and thus also the proof of the Theorem \ref{wqlTn}.
\end{proof}

\section{Weak quenched limiting distributions - averaged centering}\label{averagedcentering}

In this section we prove weak convergence with the averaged centering
stated in Theorem \ref{wqlTnA}. The argument is similar in most
respects to the proof of Theorem \ref{wqlTn} in the previous section,
so we will concentrate now on those parts of the argument that are
different. Recall that by Proposition \ref{bigtransferA} we only need
to establish a weak quenched limit for 
\be\label{sHNn}
 \s_{n,\w} = 
\begin{cases}
 H(N_{n,\w}) & \kappa \in(0,1) \\
 H(N_{n,\w})*\d_{-D'(n)} & \kappa = 1 \\
 H(N_{n,\w})*\d_{-\bar\b n^{1-1/\k}} & \kappa \in(1,2),
\end{cases}
\ee
where $H:\mathcal{M}_p \ra \mathcal{M}_1$ is given by \eqref{Hdef}. 
We will  use Proposition \ref{PPconvergence} and, once again, we have
to use a truncated version of the mapping $H$. We will use the mapping
$H_\e$ defined in \eqref{Hedef}. The following lemma, whose proof is
identical to that  of Lemma \ref{barHecont}, shows that $H_\e$ is also
continuous on the relevant subset of $\mathcal{M}_p$.  
\begin{lem}\label{Hecont}
 $H_\e$ is continuous on $\mathcal{M}_p^{(\e)} = \{\zeta \in \mathcal{M}_p : \zeta(\{\e\}) = 0 \}$. 
\end{lem}

An immediate consequence of Lemma \ref{Hecont} and Proposition
\ref{PPconvergence} is 
\be\label{twqlA}
H_\e(N_{n,\w}) \overset{Q}{\Lra} H_\e(N_{\l,\k}).
\ee
We divide the remainder of the proof of Theorem \ref{wqlTnA} into two
cases: $\k\in(0,1)$ and $\k\in[1,2)$.

\subsection{Case I: $\k \in (0,1)$}
The case $\k\in(0,1)$ is almost identical to the proof of Theorem
\ref{wqlTn}. Due to \eqref{twqlA}, it is enough to show that  
\be\label{econkl1}
 \lim_{\e\ra 0^+} H_\e(N_{\l,\k}) = H(N_{\l,\k}), \quad \Pv\text{-a.s.}
\ee
and
\be\label{ipconkl1}
 \lim_{\e\ra 0} \limsup_{n\ra\infty} Q\left( \rho(H_\e(N_{n,\w}), H(N_{n,\w})) \geq \d \right) = 0, \quad \forall \d>0. 
\ee
The proof of \eqref{econkl1} is similar to that of \eqref{econ}. The
main difference between the proof of 
\eqref{ipconkl1} and that of \eqref{ipcon} is that now we are using
the fact that for any two random variables $X$ and $Y$ defined on the
same probability space, with respective laws ${\mathcal L}_{X}$ and
${\mathcal L}_{Y}$, $\rho({\mathcal L}_{X}, {\mathcal L}_{Y})\leq
\bigl( E|X-Y|\bigr)^{1/2}$, after which one uses once again
\eqref{ETnutail} and Karamata's theorem.

\subsection{Case II: $\k \in [1,2)$}
The difference in this case is that centering is needed. Let 
\[
 c_n(\e) = 
\begin{cases}
 E_Q\left[ \b_1 \ind{\b_1 \in (\e n , \bar\nu n]} \right] & \text{ if } \k =1 \\
 n^{1-1/\k} E_Q\left[ \b_1 \ind{\b_1 > \e n^{1/\k}} \right] & \text{ if } \k \in (1,2).
\end{cases}
\]
Recalling the definitions from the statement of Proposition
\ref{bigtransferA}, we see that the tail decay of $\b_1$ implies that 
\[
 \lim_{n\ra\infty} c_n(\e) = 
\begin{cases}
 \tilde{c}_{\l,1}(\e) & \text{ if } \k = 1 \\
 c_{\l,\k}(\e)  & \text{ if } \k \in (1,2) 
\end{cases}
, \quad\text{where } \l = \k \tc. 
\]
Combining this with \eqref{twqlA} we obtain that
\be\label{ecclaim}
 H_\e(N_{n,\w}) * \d_{-c_n(\e)} \Lra 
\begin{cases}
 H_\e(N_{\l,1}) * \d_{-\tilde{c}_{\l,1}(\e)} & \text{ if } \k = 1 \\
 H_\e(N_{\l,\k}) * \d_{-c_{\l,\k}(\e)} & \text{ if } \k \in (1,2).
\end{cases}
\ee

We use, once again, \cite[Theorem 3.2]{bCOPM}. By \eqref{ecclaim}, in
the case $\k \in (1,2)$, 
weak convergence of the measures $\s_{n,\w}$ in \eqref{sHNn} will
follow once we show that
\be\label{econA2}
 H_\e(N_{\l,\k}) * \d_{-c_{\l,\k}(\e)} \quad \text{converges $\Pv$-a.s. as } \e\ra 0^+,
\ee
and 
\be\label{ipconA2}
 \lim_{\e\ra 0} \limsup_{n\ra\infty} Q\left( \rho\left( H(N_{n,\w}) * \d_{-\bar\b n^{1-1/\k}}, H_\e(N_{n,\w}) * \d_{-c_n(\e)}\right) \geq \d \right) = 0, \quad \forall \d>0. 
\ee
The argument in the case $\k=1$ is exactly the same if one replaces
every instance of $\bar\b n^{1-1/\k}$ and $c_{\l,\k}(\e)$ with $D'(n)$
and $\tilde{c}_{\l,1}(\e)$, respectively. Thus we will only give the
proof in the case $\k \in (1,2)$.  

To prove \eqref{econA2}, let $\xi_1>\xi_2>\ldots$ be the points of
$N_{\l,\k}$. By Theorem
3.12.2 in \cite{samorodnitsky:taqqu:1994}, the shifted truncated sums 
$$
\sum_{i\geq 1} \xi_i \tau_i \ind{\xi_i > \e}- c_{\l,\k}(\e)
$$
converge a.s. as $\e\ra 0^+$. The convergence above is true for almost every realization of the joint sequence $(\xi_i, \tau_i)_{i\geq 1}$, but by Fubini's theorem the same remains true for a.e. realization of the Poisson process $N_{\l,\k}$. Since
a.s. convergence implies weak convergence, we obtain \eqref{econA2}. 

Turning now to the proof of \eqref{ipconA2}, we use the same upper
bound on the Prohorov's distance as in the proof of Theorem
\ref{wqlTn}. Since 
$\bar{\b}n^{1-1/\k} - c_n(\e) = n^{1-1/\k} E_Q[ \b_1
  \ind{\b_1/n^{1/\k} \leq \e} ]$, we have
\begin{align*}
& \rho\left( H(N_{n,\w}) * \d_{-\bar\b n^{1-1/\k}}, H_\e(N_{n,\w}) *
\d_{-c_n(\e)}\right) \\
& \leq \left( \frac{2}{n^{2/\k}}
\left(\sum_{i=1}^n \bigl\{ \beta_i \ind{ \b_i/n^{1/\k} \leq \e } - E_Q[ \b_1
  \ind{\b_1/n^{1/\k} \leq \e} ]\bigl\}\right)^2\right)^{1/3} \\
&\qquad + \left( \frac{2}{n^{2/\k}} \sum_{i=1}^n  \beta_i^2 \ind{
  \b_i/n^{1/\k} \leq \e }\right)^{1/3}. 
\end{align*}
Therefore,
\begin{align}
 & \limsup_{n\to\infty} Q\left( \rho\left( H(N_{n,\w}) * \d_{-\bar\b
  n^{1-1/\k}}, H_\e(N_{n,\w}) * \d_{-c_n(\e)}\right)\geq \d \right) \nonumber \\
& \leq \limsup_{n\to\infty} Q\left( \frac{2}{n^{2/\k}} \left(\sum_{i=1}^n \bigl\{ \beta_i \ind{ \b_i/n^{1/\k} \leq \e } - E_Q[ \b_1
  \ind{\b_1/n^{1/\k} \leq \e} ]\bigl\}\right)^2 \geq \frac{\d^3}{8} \right) \label{betasmall}\\
& \qquad + \limsup_{n\to\infty} Q\left( \frac{2}{n^{2/\k}} \sum_{i=1}^n  \beta_i^2 \ind{
  \b_i/n^{1/\k} \leq \e } \geq \frac{\d^3}{8} \right) \label{betasquared}. 
\end{align}
Lemma \ref{ble} implies that \eqref{betasmall} vanishes as $\e\ra 0$, and (as in the proof of Theorem \ref{wqlTn}) Markov's inequality, \eqref{ETnutail} and Karamata's
theorem imply that \eqref{betasquared} vanishes as $\e \ra 0$ as well. 
This completes the proof of a limiting distribution for
$\s_{n,\w}$, and the proof of Theorem \ref{wqlTnA} follows by an
application of Proposition \ref{bigtransferA}.

\section{Converting from time to space}\label{Tn2Xn}

In this section we show that the  weak quenched limit theorem for the 
hitting times $T_n$ in Theorem \ref{wqlTnA}  implies the 
weak quenched limit theorem for the random walk $X_n$ in Corollary
\ref{wqlXn}.

For any $t\geq 0$, let 
\[
 X_t^* = \max \{ X_k : k\leq t \} = \max \{ n \in \Z: T_n \leq t \}
\]
be the farthest the random walk has traversed to the right by time $t$. The usefulness of $X_t^*$ stems from the identity of the events
\be\label{eventeq}
 \{ X_t^* < x \} = \{ T_x > t \} \quad \text{and} \quad \{ X_t^* \geq x \} = \{ T_x \leq t \}. 
\ee
The following lemma implies that $X_n$ typically is very close to $X_n^*$. 
\begin{lem}\label{XXstarlem}
 Let Assumptions \ref{iidasm} and \ref{tasm} hold. Then, $\limsup_{n\ra\infty} \frac{X_n^* - X_n}{\log n} <\infty$, $\P$-a.s.
\end{lem}
\begin{proof}
 The event $\{X_n^* - X_n \geq M\}$ implies that for some $x=0,1,\ldots n-1$ the random walk after first hitting $x$ then backtracks to $x-M$. Thus,
\[
 \P( X_n^* - X_n \geq M) \leq \sum_{x=0}^{n-1} \P^x(T_{x-M}<\infty) = n \P(T_{-M}<\infty),
\]
where the last equality follows from the translation invariance of the
measure $P$ on environments. It was shown in \cite[Lemma 3.3]{gsMVSS}
that Assumptions \ref{iidasm} and \ref{tasm} imply that there exist
constants $C,C'>0$ such that $\P(T_{-M}<\infty) \leq C e^{-C'
M}$. Taking $M=K \log n$ for $K>2/C'$ we obtain that 
\[
 \P( X_n^* - X_n \geq \d (\log n)^2) \leq C n^{-(C'K-1)}, 
\]
which is  summable over $n$. The claim of the lemma now follows from
the Borel-Cantelli Lemma.
\end{proof}

We will also need the following Corollary of Theorem \ref{wqlTnA}. 
\begin{cor}\label{wqlTnAproj}
Let $\k \in (0,2)$, and let $\mu_{\l,\k}$ be the limiting random probability measure given by the conclusion of Theorem \ref{wqlTnA} (that is $\mu_{n,\w} \Lra \mu_{\l,\k}$). 
Then, $\mu_{n,\w}(x,\infty) \Lra \mu_{\l,\k}(x,\infty)$ for any $x\in\R$. 
\end{cor}
\begin{proof}
First of all, note that the random probability measures $\mu_{\l,\k}$ are continuous distributions with probability 1. That is, $\Pv( \mu_{\l,\k}(\{x\}) > 0 ) = 0$. 
To see this, note that on an event of probability 1,  we can write $\mu_{\l,\k} = E_1(\cdot/\xi_1) \ast 
\tilde{\mu}_{\l,\k}$, where $\xi_1$ is the largest point of
the Poisson process, $E_1$ is the standard exponential distribution,
and $\tilde{\mu}_{\l,\k}$ is another random probability distribution. The
continuity of the exponential distribution then implies that $\mu_{\l,\k}$ is also continuous.  

For any $x\in\R$, the mapping $\pi \mapsto \pi(x,\infty)$ from
$\mathcal{M}_1$ to $\R$ is continuous on the set $\mathcal{C}_x = \{
\pi \in  \mathcal{M}_1: \, \pi(\{x\}) = 0 \}$. Since we showed above
that $P(\mu_{\l,\k} \in \mathcal{C}_x ) = 1$, the continuous mapping
theorem implies that $\mu_{n,\w}(x,\infty) \Lra \mu_{\l,\k}(x,\infty)$
as $n\ra\infty$. 
\end{proof}

We are now ready to give the proof of Corollary \ref{wqlXn}. 
\begin{proof}[Proof of Corollary \ref{wqlXn}]
We will first prove Theorem \ref{wqlXn} with $X_n^*$ in place of $X_n$ and then use Lemma \ref{XXstarlem} to transfer the results to $X_n$. 
Since the centering and scaling used depends on $\k$ we divide the proof into three cases: $\k\in(0,1)$, $\k=1$, and $\k\in(1,2)$.

\subsection{Case I: $\k \in (0,1)$}
If $\k \in (0,1)$, then \eqref{eventeq} implies that for $x \in \R$ fixed 
\begin{align*}
 P_\w\left( X_n^* < x n^\k \right) &= P_\w \left( T_{ \lceil x n^\k \rceil } > n \right) \\
&= P_\w \left( \frac{ T_{ \lceil x n^\k \rceil } }{ \lceil x n^\k \rceil^{1/\k} } > \frac{n}{ \lceil x n^\k \rceil^{1/\k} } \right) \\
&= \mu_{\lceil x n^\k \rceil,\w} \left( \frac{n}{ \lceil x n^\k \rceil^{1/\k} }, \infty \right)
\end{align*}
Corollary \ref{wqlTnAproj} implies that the last term above converges in distribution to $\mu_{\l,\k}(x^{-1/\k},\infty) = H(N_{\l,\k})(x^{-1/\k},\infty)$ (note that here we are using the monotonicity of distribution functions, the fact that $\mu_{\l,\k}$ is a continuous distribution with probability 1, and the fact that $n/\lceil x n^\k \rceil^{1/\k} \ra x^{-1/\k}$ as $n\ra\infty$). Thus, we have shown that
\be\label{wqlXc1}
P_\w\left( X_n^* < x n^\k \right)  \Lra  H(N_{\l,\k})(x^{-1/\k},\infty). 
\ee

Next, note that $X_n \leq X_n^*$ implies that 
\be\label{XXstarublb}
 P_\w(X_n^* < x n^\k) \leq P_\w(X_n < x n^\k) \leq P_\w(X_n^* < x n^\k + (\log n)^2 ) + P_\w(X_n^* - X_n > (\log n)^2). 
\ee
Lemma \ref{XXstarlem} implies that $ P_\w( X_n^* - X_n > (\log n)^2 )$ converges to 0 in $L^1$, and thus also in distribution. 
Therefore, \eqref{wqlXc1} and \eqref{XXstarublb} complete the proof of Theorem \ref{wqlXn} when $\k\in (0,1)$ (here we again are using the monotonicity of distribution functions and the fact that $\mu_{\l,\k} = H(N_{\l,\k})$ is continuous with probability 1).

\subsection{Case II: $\k=1$}
Recall from Remark \ref{DDprime} that the sequence $D(n)$ is given by 
\[
 D(n) = \frac{\fl{n/\bar{\nu}}}{n} D'( \fl{n/\bar{\nu}}) = \frac{\fl{n/\bar{\nu}}}{n} E_Q\left[ \b_1 \ind{\b_1 \leq \bar{\nu}\fl{n/\bar{\nu}}} \right].  
\]
Note first of all that this implies $D(n) \sim A \log n$, where $A=\tc /\bar{\nu}$. Moreover, this explicit representation also gives that
$D(y(n)) - D(x(n)) \ra 0$ as $n\ra\infty$ for any sequences $x(n),y(n) \ra\infty$ with $x(n) \sim y(n)$.

We postpone for now the definition of the averaged centering term $\d(n)$ for the random walk $X_n$. Whatever $\d(n)$ is, for fixed $x$ we let $\gamma(n) = \lceil \d(n) + x n/(\log n)^2 \rceil$. 
Then, \eqref{eventeq} implies that
\begin{align}
 P_\w\left( \frac{X_n^* - \d(n)}{n/(\log n)^2} < x \right) &= P_\w\left( X_n^* < \d(n) + x n/(\log n)^2 \right) \nonumber \\
&= P_\w \left( T_{  \gamma(n) } > n \right) \nonumber \\
&= P_\w \left( \frac{ T_{  \gamma(n)  } -  \gamma(n)  D\left(  \gamma(n)  \right) }{  \gamma(n)  } > \frac{n -  \gamma(n)  D\left(  \gamma(n)  \right) }{  \gamma(n)  } \right) \nonumber \\
&= \mu_{ \gamma(n) ,\w} \left( \frac{n -  \gamma(n)  D\left(  \gamma(n)  \right) }{  \gamma(n)  }, \infty \right). \label{pistarmu}
\end{align}
Now, we can choose $\d(n)$ so that 
\be\label{choosedelta}
 \d(n) D(\d(n)) = n + o(1), \quad \mbox{as $n\ra\infty$}. 
\ee
Then, recalling the definition of $\gamma(n)$ and the fact that $D(n) \sim A\log n$ as $n\ra\infty$, this implies that 
\[
 \gamma(n) \sim \d(n) \sim \frac{n}{A \log n}, \quad \mbox{as $n\ra\infty$}, 
\]
and 
\[
 \lim_{n\ra\infty} \frac{n - \gamma(n) D(\gamma(n))}{\gamma(n)} = -A^2 x. 
\]
(Note that in this last limit we used the fact that $D(\gamma(n))-D(\d(n)) \ra 0$ since $\d(n),\gamma(n) \ra \infty$ and $\d(n) \sim \gamma(n)$ as $n\ra\infty$). 

Recalling \eqref{pistarmu} and having chosen $\d(n)$ according to \eqref{choosedelta}, Corollary \ref{wqlTnAproj} implies that 
\[
 P_\w\left( \frac{X_n^* - \d(n)}{n/(\log n)^2} < x \right) \Lra \lim_{\e\ra 0^+} \left(H_\e(N_{\l,1}) * \d_{-c_{\l,1}(\e)} \right)(-A^2 x, \infty), \quad \forall x\in\R.
\]
Replacing $X_n^*$ with $X_n$ in the above statement is again accomplished by using Lemma \ref{XXstarlem}. The proof is essentially the same as in the case $\k\in(0,1)$ and is therefore ommitted. 

\subsection{Case III: $\k \in (1,2)$}
Let $x \in \R$ be fixed, and define $\psi(n) = \lceil n \vp + x n^{1/\k} \rceil$. 
Then \eqref{eventeq} implies that 
\begin{align*}
 P_\w\left( \frac{X_n^* - n\vp}{n^{1/\k}} < x \right) &= P_\w\left( X_n^* < n \vp  + x n^{1/\k} \right) \\
&= P_\w \left( T_{ \psi(n) } > n \right) \\
&= P_\w \left( \frac{ T_{ \psi(n) } - \psi(n)/\vp }{ \psi(n)^{1/\k} } > \frac{n - \psi(n)/\vp }{ \psi(n)^{1/\k} } \right) \\
&= \mu_{\psi(n),\w} \left( \frac{n - \psi(n)/\vp }{ \psi(n)^{1/\k} }, \infty \right)
\end{align*}
Note that 
\[
 \lim_{n\ra\infty} \frac{n - \psi(n)/\vp }{ \psi(n)^{1/\k} } = \lim_{n\ra\infty} \frac{n - \lceil n \vp  + x n^{1/\k} \rceil/\vp }{ \lceil n \vp  + x n^{1/\k} \rceil^{1/\k} } = -x \vp^{-1-1/\k}, 
\]
and thus Corollary \ref{wqlTnAproj} implies that
\[
 P_\w\left( \frac{X_n^* - n\vp}{n^{1/\k}} < x \right) \Lra \lim_{\e \ra 0^+} \left(H_\e (N_{\l,\k}) * \d_{-c_{\l,\k}(\e)}\right)(-x \vp^{-1-1/\k}, \infty), \quad \forall x\in\R. 
\]
We again omit the proof that $X_n^*$ can be replaced by $X_n$ in the above statement. 
\end{proof}

\appendix

\section{Proof of Lemma \ref{l:meas}}\label{Hmeas}

The easiest way to see the measurability of $\bar{H}$ is to represent
it as a composition of two maps, and to show that each one of these
maps is measurable. We write $\bar{H} = 
\bar{H}_2\circ \bar{H}_1$, where $\bar{H}_1:\,
\mathcal{M}_p \to l^2$ is defined by 
$$
 \bar{H}_1(\zeta)(\cdot) = 
\begin{cases}
(x_{(1)}, x_{(2)},\ldots) & \text{if} \ 
\sum_{i\geq 1} x_i^2 < \infty \\ 
{\bf 0} & \text{otherwise}, 
\end{cases}
$$
where $x_{(1)}\geq x_{(2)}\geq \ldots$ is the nonincreasing
rearrangement of the points of $\zeta = \sum_{i\geq 1} \d_{x_i}$, and
${\bf 0}$ is the zero element in $l^2$, while $\bar{H}_2:\,  l^2\to
\mathcal{M}_1$ is defined by 
$$
 \bar{H}_2({\bf x})(\cdot) = 
\Pv\left( \sum_{i\geq 1} x_i(\tau_i -1) \in \cdot \,  \right)
$$
for ${\bf x}= (x_1,x_2,\ldots)\in l^2$, where $\tau_i$ are i.i.d.\ Exp(1) random variables under the measure $\Pv$. Since the Borel $\sigma$-field
on $l^2$ coincides with its cylindrical $\sigma$-field, measurability
of the map $ \bar{H}_1$ will follow once we check both that for each
$k=1,2,\ldots$ the map $\bar{H}_{1,k}:\, \mathcal{M}_p \to \R$ defined
for $\zeta = \sum_{i\geq 1} \d_{x_i}$ by $\bar{H}_{1,k}(\zeta) =
x_{(k)}$ is measurable, and also that the set 
$$
F= \Bigl\{ \zeta = \sum_{i\geq 1} \d_{x_i}:\ \sum_{i\geq 1}
x_i^2<\infty\Bigr\}
$$
is a measurable subset of $\mathcal{M}_p$. The first statement follows
since each $\bar{H}_{1,k}$ is, clearly, a continuous map. The second
statement follows by writing $F=\cup_{m=1}^\infty F_m$, where for each
$m$, 
$$
F_m= \Bigl\{ \zeta = \sum_{i\geq 1} \d_{x_i}:\ \sum_{i\geq 1}
x_{(i)}^2\leq m\Bigr\}
$$
is, by the continuity of the maps $\bar{H}_{1,k}$ and Fatou's lemma, a
closed set. 

In order to prove measurability of the map $ \bar{H}_2$, it is enough
to prove its continuity. Let ${\bf x}^{(n)}=(x_1^{(n)},x_2^{(n)},
\ldots)$, $n=1,2,\ldots$ be a sequence in $l^2$ converging to ${\bf
  y}= (y_1,y_2,\ldots)\in l^2$. Instead of proving that $\sum_{i\geq
  1} x_i^{(n)}(\tau_i -1)$ converges weakly to $\sum_{i\geq 1}
y_i(\tau_i -1)$ it is, of course, sufficient to prove convergence in
probability. This latter convergence follows immediately because
$$
\Ev \left( \sum_{i\geq   1} x_i^{(n)}(\tau_i -1) - \sum_{i\geq 1}
y_i(\tau_i -1)\right)^2 = \| {\bf x}^{(n)}-{\bf y}\|_2^2\,.
$$

\section{Proof of Lemma \ref{SFtails}}\label{details}

The tail decay of $E_\w S$ was analyzed in \cite{p1LSL2}, but for
completeness we will briefly outline the argument here. 
By using $h$-transforms one can compute a formula for the transition
probabilities of the random walk conditioned on exiting the interval
$(0,\nu)$ to the right. Given these conditional transition
probabilities one can apply the formula \eqref{QETform} for the
quenched expectation of the amount of time to move one step to the
right. Before giving this formula we need to introduce some
notation. Recall that $M_1 = \max\{ \Pi_{0,j}:\, 0\leq j<\nu \}$. Let
$i_0 = \max \{ i\in[1,\nu]: \Pi_{0,i-1} = M_1 \}$,  and denote 
\[
 M^- = \min \{ \Pi_{i,j} \, : \, 0 < i \leq j < i_0 \} \wedge 1, \quad\text{and}\quad M^+ = \max \{ \Pi_{i,j} \,:\, i_0 < i \leq j < \nu \} \vee 1. 
\]
Then, following the proof of Corollary 4.2 in \cite{p1LSL2}, one can show that for any $0<i<\nu$, 
\be\label{conditionright}
 E_\w^i\left[T_{i+1} \, \bigl| \, T_\nu < T_0 \right] \leq 1 + \frac{2\nu^3 M^+}{(M^-)^3} \leq \frac{3\nu^3 M^+}{(M^-)^3}.  
\ee
This immediately implies that $E_\w S \leq \frac{3\nu^4 M^+}{(M^-)^3}$. 
The proof of the tail decay \eqref{Stail} of $E_\w S$ is then accomplished by recalling \eqref{nutail} and the following Lemma from \cite{p1LSL2}.
\begin{lem}[Lemma 4.1 in \cite{p1LSL2}]\label{Mpmlem}
 For any $0<\e<1$ and $\e',\d>0$, 
\[
 Q( M^+ > n^\d, \, M_1 > n^{(1-\e)/\k} ) = o( n^{-1+\e-\d\k+\e'}),
\]
and
\[
 Q(M^- < n^{-\d}, \, M_1 > n^{(1-\e)/\k} ) = o( n^{-1+\e-\d\k+\e'}).
\]
\end{lem}
Applying this lemma and recalling from \eqref{nutail} that $\nu$ has
exponential tails, we obtain that for any $0<\e<1$ and $\e',\d>0$, 
\begin{align*}
 Q\left( E_\w S > n^{5\d}, \, M_1 > n^{(1-\e)/\k} \right) 
&\leq Q(  \nu^4 > n^{\d} ) + Q(M^+ > n^{\d}, \, M_1 > n^{(1-\e)/\k}) \\
&\quad + Q(M^- < n^{-\d}, \, M_1 > n^{(1-\e)/\k}) \\
&= o(n^{-1+\e-\d \k + \e'}).
\end{align*}
Choosing $5\d =6\e/\k$ completes the proof of \eqref{Stail}.

The proof of \eqref{Ftail} is similar. We note first of all that 
\begin{align*}
 E_\w F^{(1)} &= 1  + E_\w^{-1}\left[ T_0 \right]P_\w\left( X_1 = -1 \, | \, T_0^+ < T_\nu \right) + E_\w^1\left[ T_0 \, | \, T_0 < T_\nu \right]P_\w\left( X_1 = 1 \, | \, T_0^+ < T_\nu \right)\\
&\leq 1  + E_\w^{-1}\left[ T_0 \right] + E_\w^1\left[ T_0 \, | \, T_0 < T_\nu \right] \\
&= 2+2W_{-1} + E_\w^1\left[ T_0 \, | \, T_0 < T_\nu \right].
\end{align*}
It was shown in \cite[Lemma 2.2]{pzSL1} that $W_{-1}$ has exponential
tails under the measure $Q$, so we only need to anlayze the tails of
the $E_\w^1\left[ T_0 \, | \, T_0 < T_\nu \right]$.  
To this end, the proof of \eqref{conditionright} can be modified by
instead conditioning on exiting the interval $(0,\nu)$ to the left in
order to obtain that 
\[
 E_\w^i\left[T_{i-1} \, \bigl| \, T_0 < T_\nu \right] \leq \frac{3\nu^3 (M^+)^3}{M^-}, \quad \text{ for any } 0<i<\nu.
\]
Then, as was done above for $E_\w S$, we can use \eqref{nutail} and
Lemma \ref{Mpmlem} to obtain that for any $0<\e<1$ and $\e',\d>0$,  
\[
 Q\left( E_\w^1\left[ T_0 \, | \, T_0 < T_\nu \right] > n^{5\d}, \,
 M_1 > n^{(1-\e)/\k} \right) = o(n^{-1+\e-\d \k + \e'}). 
\]
Choosing again $5\d =6\e/\k$ proves \eqref{Ftail}.



\bibliographystyle{plain}  
\bibliography{RWRE}

\end{document}